\documentclass[12pt,twoside,english,a4paper]{article}
\usepackage{babel,amssymb,amsthm}
\usepackage{graphicx}
\usepackage{amsmath}
\usepackage{bm}
\usepackage{float, framed}
\usepackage{wrapfig}

\newcommand{\punto}{\,\cdot\,}

\newcommand{\smallfrac}[2]{{\textstyle\frac{#1}{#2}}} 
\newcommand{\jump}[1]{[\![#1]\!]}
\newcommand{\ave}[1]{\{\!\!\{#1\}\!\!\}}
\newcommand{\triple}[1]{|\!|\!|#1|\!|\!|}

\newcommand{\bs}{\boldsymbol}
\newcommand{\bff}{\mathbf}

\newtheorem{proposition}{Proposition}[section]
\newtheorem{corollary}[proposition]{Corollary}

\newtheorem{theorem}[proposition]{Theorem}

\numberwithin{equation}{section}

\title{Boundary and coupled boundary-finite element methods for transient wave-structure interaction}

\date{\today}

\author{George C. Hsiao, Tonatiuh S\'anchez-Vizuet
 \& Francisco--Javier Sayas\footnote{TSV and FJS partially funded by NSF grant DMS 1216356.}  \\
Department of Mathematical Sciences, University of Delaware, USA\\
{\tt \{ghsiao, tonatiuh, fjsayas\}@udel.edu }}

\begin{document}

\maketitle

\begin{abstract}
We propose time-domain boundary integral and coupled boundary integral and variational formulations for acoustic scattering by linearly elastic obstacles. Well posedness along with stability and error bounds with explicit time dependence are established. Full discretization is achieved coupling boundary and finite elements; Convolution Quadrature is used for time evolution in the pure BIE formulation and combined with time stepping in the coupled BEM/FEM scenario. Second order convergence in time is proven for BDF2-CQ and numerical experiments are provided for both BDF2 and Trapezoidal Rule CQ showing second order behavior for the latter as well. \\
{\bf AMS Subject classification.} 65R20, 65M38, 74J20, 74F10.\\
{\bf Keywords.} Time-Domain Boundary Integral Equations, Convolution Quadrature, Scattering, Linear Elasticity, Coupling FEM-BEM
\end{abstract}

%
\section{Introduction}
%
The study of the acoustic scattering by a linearly elastic obstacle and its corresponding elastic response has been a subject of interest in both the mathematical and engineering community for some years now. In the case of time-harmonic regime, the study of the existence and uniqueness of solutions dates back at least to 1986 \cite{HaJe:1986}. The well-posedness of several purely boundary integral formulations was analized in \cite{HsKlSc:1988,Hsiao:1994,LuMa:1995} by assuming that the scatterer had at least a boundary of class $\mathcal C^2$. Later on, in the 2000's, combined boundary integral and variational formulations were proposed in \cite{HsKlRo:2000,ElHsRa:2008} and proven to be well posed also for smooth scatterers. In these works the elastic response is modeled variationally and the unbounded acoustic scattering is treated with either a boundary integral equation or by introducing an artificial boundary where an absorbing condition is imposed. Recently, the more general case of a Lipschitz scatterer was dealt with in \cite{BaDjEs:2014a} within the framework of a variational formulation with a fictitious boundary.

On the computational side, the coupling of Boundary Elements and $hp$-Finite Elements was studied in \cite{DeOd:1996} where the Burton-Miller equation is used to model the acoustic wavefield; the authors provide a posteriori error bounds aiming for an adaptive implementation. The fictitious boundary approach with finite elements has been thoroughly investigated in \cite{MaMeSe:2004,GaMaMe:2012,GaMaMe:2014} and a DG-like implementation was carried out recently in \cite{BaDjEs:2014b}.

The transient regime, on the other hand, has not enjoyed so much attention --at least in the mathematical community-- as its frequency-domain cousin. In \cite{Feng:2000,FeXi:2004} the problem is posed in a slab-like unbounded domain imposing first order absorbing boundary conditions, while in \cite{HsSaWe:2015} well-posedness is established for the coupled boundary integral/variational formulation also in a slab-like region.  Within the engineering community, the time-domain case has attracted attention at least since 1991. Numerous approaches have been attempted without much theoretical justification but with satisfactory results. To cite some examples, BE/FE coupling with Convolution Quadrature was employed in \cite{AnEs:1991}, BE/BE coupling using Newmark time integration was the preferred treatment in \cite{MaSo:2006} and FE/FE coupling with an absorbing boundary condition and Newmark time integration were used in \cite{FlKaWo:2006}. A comprehensive list of related work can be found in \cite{Soares:2011}.

The present work strives to fill the gap in the mathematical analysis of the time-domain wave-structure problem. It deals both with the pure boundary integral formulation --which arises naturally when dealing with homogeneous acoustic and elastic domains-- and the combined boundary integral/variational treatment where integral equations are used only for the acoustic dynamics, being better suited for general elastic scatterers. The former case leads to a numerical treatment where only Boundary Elements are used for space discretizations while the latter is naturally suited for a coupled Boundary Element/Finite Element implementation. 

Despite the fact that each formulation requires a very different numerical discretization, the techniques and tools required to carry out the theoretical study are surprisingly similar. Following \cite{LaSa:2009a,Sayas:2013,BaLaSa:2014} the analysis is done in the Laplace-domain aiming for a Convolution Quadrature treatment similar to that done for the purely acoustic case in \cite{BaLuSa:2015,FaMoSc:2012,FaMoSc:2014}. We deal simultaneously with the continuous and discrete cases by posing the problems in a general closed subspace of the appropriate function spaces. Well-posedness is proved in the Laplace-domain via an equivalent exotic transmission problem for which a variational formulation is found. The resulting stability bounds are written carefully in terms of the Laplace parameter $s$ in order to apply results from \cite{Sayas:2014} which give explicit time-domain estimates. Error bounds in the time-domain are obtained following a similar approach for the semidiscrete problem. Full discretization and convergence estimates are given for the case of BDF2-CQ.

Finally, numerical experiments are carried out using Boundary Elements coupled with BDF2 and Trapezoidal Rule Convolution Quadrature for the BIE formulation and BE/FE with Trapezoidal Rule Convolution Quadrature for the boundary element part paired with Trapezoidal Rule time stepping for the elastic field. The results support the order of convergence predicted in the theoretical part of the paper. 

\paragraph{A word on real and complex Sobolev spaces.} Basic knowledge of Sobolev spaces $H^1(\Omega)$, its trace space $H^{1/2}(\partial\Omega)$ and its dual $H^{-1/2}(\partial\Omega)$ is assumed throughout. 
In everything that follows, the Sobolev spaces $H^1(\Omega)$ and $H^{\pm1/2}(\partial\Omega)$ will be used with the same notation for  real-valued and complex-valued functions. If, for instance, we take a subspace $V \subset H^{1/2}(\partial\Omega)$, we will understand that it is a subspace of the real-valued space $H^{1/2}(\partial\Omega)$, and that its complexification will be used whenever complex values are considered. Scripted parentheses will be used for real $L^2$ inner products of scalar, vector- or matrix-valued functions:
\[
(a,b)_\Omega:=\int_\Omega a\,b,
	\qquad
(\mathbf a,\mathbf b)_\Omega:=\int_\Omega\mathbf a\cdot\mathbf b,
	\qquad
(\mathbf A,\mathbf B)_\Omega:=\int_\Omega \mathbf A:\mathbf B,	
\]
where in the latter the colon denotes the Frobenius inner product of matrices. When the functions take complex values, we will still use the parenthesis with this precise meaning and will explicitly conjugate quantities whenever needed. When we change the font describing a Sobolev or Lebesgue space from italic to boldface we mean the product space of $d$ copies of the space. For example, $\mathbf L^2(\Omega):=L^2(\Omega)^d$.
%
\section{Homogeneous isotropic solids: BIE formulation}\label{sec:2}
%
\paragraph{The PDE system.}
Let $\Omega_-\subset \mathbb R^d$ be a bounded, not necessarily connected region, lying on one side of its Lipschitz boundary $\Gamma$, and let $\Omega_+:=\mathbb R^d\setminus\overline{\Omega_-}$ be its unbounded complement. The unit normal vector field on $\Gamma$, exterior to $\Omega_-$, will be denoted $\bs\nu$. Our problem can be explained as follows: an incident acoustic field $v^{\mathrm{inc}}$ arrives at an obstacle at time $t=0$ and interacts with a homogeneous isotropic elastic body occupying $\Omega_-$. The elastic properties of the material are represented by the two Lam\'e parameters that define the structure of the linear stress tensor:
\[
\bs\sigma(\bff u):=2\mu \bs\varepsilon(\bff u)+\lambda(\nabla\cdot\bff u)\bff I, 
\qquad 
\bs\varepsilon(\bff u):=\smallfrac12 (\nabla\bff u+(\nabla \bff u)^\top),
\]
where $\bff I$ is the $d\times d$ identity matrix and $\bff u:\Omega_-\to\mathbb R^d$ is the displacement field. This will be the only occurrence of the Lam\'e parameters (until we get to the section on numerical experiments), and we will use the Greek letters $\lambda,\mu$ to represent other quantities. Related to the stress tensor we can define the Lam\'e operator $\Delta^*\bff u:=\nabla\cdot\bs\sigma(\bff u)$ and the normal traction on the boundary by
$\bff t(\bff u):=\bs\sigma(\bff u)\bs\nu.$ Before we give a rigorous mathematical formulation of the problem, let us start by writing the system of PDE with transmission conditions that we want to solve:
\begin{subequations}\label{eq:2.1}
\begin{alignat}{6}
\rho_\Sigma \bff u_{tt} =& \,\Delta^* \bff u  & \qquad & \mbox{in $\Omega_-\times [0,\infty)$},\\
c^{-2} v_{tt}=&\,\Delta v & \qquad & \mbox{in $\Omega_+\times [0,\infty)$},\\
-\bff u_t\cdot\bs\nu =& \,\partial_\nu (v+v^{\mathrm{inc}}) 
 	& \qquad & \mbox{on $\Gamma\times [0,\infty)$},\\
\bff t(\bff u)=&\,-\rho_f (v+v^{\mathrm{inc}})_t\bs\nu
 	& \qquad & \mbox{on $\Gamma\times [0,\infty)$}.
\end{alignat}
\end{subequations}
Here $\rho_f$ and $\rho_\Sigma$ are the respective constant densities of the fluid and elastic media, the $t$ subscript denotes partial differentiation with respect to time, and $\partial_\nu$ is the normal derivative operator on $\Gamma$. This system is complemented with homogeneous initial conditions for $\bff u$ and $v$ (and their time derivatives), and a causality condition that can be expressed as: for all $t>0$, $v\equiv 0$ except in a bounded region (that changes with $t$). A derivation of this model can be found in \cite{Ihlenburg:1998}.\\

\begin{figure}[htb]
\centering
\includegraphics[scale=.55]{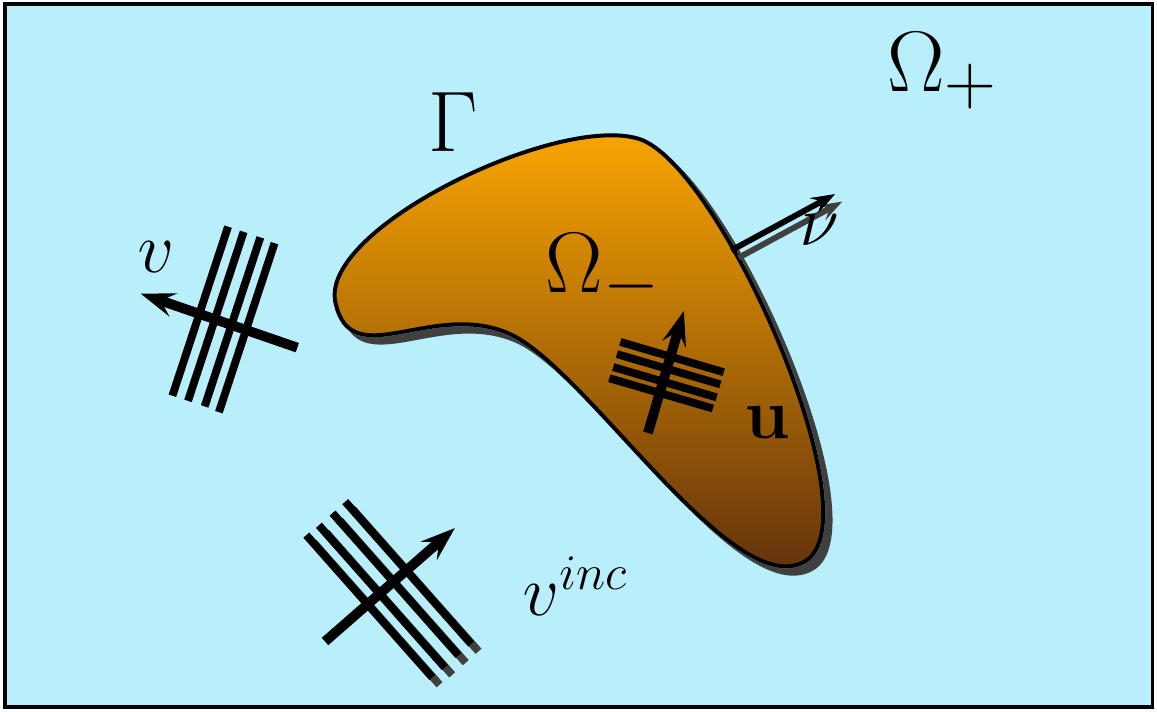}
\caption{A cartoon of the geometric setting: the solid (brown) is surrounded by the unbounded medium  (blue)}
\end{figure}

\paragraph{Traces on the boundary.} In order to properly define our concept of solution to \eqref{eq:2.1}, we will need some additional notation. Given $\bff u\in \bff H^1(\mathbb R^d\setminus\Gamma)$, we consider its interior, exterior, averaged, and difference traces:
\[
\gamma^-\bff u, \quad \gamma^+\bff u, \quad \ave{\gamma\bff u}:=\smallfrac12(\gamma^-\bff u+\gamma^+\bff u),
\quad
\jump{\gamma\bff u}:=\gamma^-\bff u-\gamma^+\bff u.
\] 
For $\bff u\in \bff H^1(\mathbb R^d\setminus\Gamma)$ such that $\bs\sigma(\bff u)\in L^2(\mathbb R^d\setminus\Gamma)^{d\times d}$ we define the weak interior and exterior traction fields using Betti's formula
\[
\langle\bff t^\mp(\bff u),\gamma \bff v\rangle_\Gamma:=
	\pm (\bs\sigma(\bff u),\bs\varepsilon(\bff v))_{\Omega_\mp}\pm (\Delta^*\bff u,\bff v)_{\Omega_\mp}
\qquad\forall\bff v\in \bff H^1(\mathbb R^d).
\]
These are elements of the dual space $\bff H^{-1/2}(\Gamma)$. From now on the $\Gamma$-subscripted angled bracket will be used to denote the duality product of $H^{-1/2}(\Gamma)$ with $H^{1/2}(\Gamma)$ or $\bff H^{-1/2}(\Gamma)$ and $\bff H^{1/2}(\Gamma)$. We will also define the average and jump of the normal traction as $\ave{\bff t(\bff u)}:=\frac12(\bff t^+(\bff u)+\bff t^-(\bff u))$ and $\jump{\bff t(\bff u)}:=\bff t^-(\bff u)-\bff t^+(\bff u)$. Two sided traces (and normal derivatives) for scalar functions (with Laplacian in $L^2$) are similarly defined. The following two operators related to the normal vector field
\[
\begin{array}{rcl}
\mathrm N : \bff H^{1/2}(\Gamma) & \longrightarrow & H^{-1/2}(\Gamma)\\
\bs\phi &\longmapsto & \bs\phi\cdot\bs\nu,
\end{array}
\qquad
\begin{array}{rcl}
\mathrm N^t :  H^{-1/2}(\Gamma) & \longrightarrow & \bff H^{1/2}(\Gamma)\\
\phi &\longmapsto & \phi\, \bs\nu,
\end{array}
\]
will be used to give rigorous meaning to some elements appearing in the transmission conditions.
\paragraph{Weak form.} For the sake of completeness we will now give the weak form of the equations \eqref{eq:2.1}. Note that all the estimates that we will produce will be developed using the Laplace transformed equations, and it will be only those equations that we will need to deal with rigorously. We look for a pair $(\bff u,v)$ of causal distributions, with values in the space
\[
\{\bff u\in \bff H^1(\Omega_-)\,:\, \Delta^*\bff u\in \mathbf L^2(\Omega_-)\}
\times
\{ v\in H^1(\Omega_+)\,:\, \Delta v\in L^2(\Omega_+)\}
\]
such that
\begin{subequations}\label{eq:2.2}
\begin{alignat}{6}
\label{eq:2.2a}
\rho_\Sigma \ddot{\bff u} =&\,\Delta^* \bff u  
		& \qquad & \mbox{(in $\bff L^2(\Omega_-)$)},\\
\label{eq:2.2b}
c^{-2} \ddot v=&\,\Delta v 
		& \qquad & \mbox{(in $L^2(\Omega_+)$)},\\
\label{eq:2.2c}
-\gamma^-\dot{\bff u}\cdot\boldsymbol\nu =&\, \partial_\nu^+ v+\alpha_0 
		 & \qquad & \mbox{(in $H^{-1/2}(\Gamma)$)},\\
\label{eq:2.2d}
\bff t^-(\bff u)=&\,-\rho_f (\gamma^+\dot v +\dot\beta_0)\boldsymbol\nu
		 & \qquad & \mbox{(in $H^{1/2}(\Gamma)$)}.
\end{alignat}
\end{subequations}
In \eqref{eq:2.2}, the upper dots are used for distributional differentiation and the parentheses in the right-hand sides tell where the distributions are compared. Full details on how to understand wave equations in the sense of vector-valued distributions can be found in \cite{Sayas:2014}. Also, we have used $\alpha_0$ and $\beta_0$  to denote general causal distributions with values in $H^{-1/2}(\Gamma)$ and $H^{1/2}(\Gamma)$ respectively. Existence and uniqueness of solution to \eqref{eq:2.2} can be proved with some additional constraints: we have to assume that the data and the solution are Laplace transformable with Laplace transforms defined in a subset of the form $\mathrm{Re}\,s>\sigma_0$ for some $\sigma_0$.
\paragraph{Laplace-transformed system.} Let us now consider a slightly different problem. Now $\lambda_0\in H^{-1/2}(\Gamma)$ and $\phi_0\in H^{1/2}(\Gamma)$ are data, and we look for $(\bff u,v)\in \bff H^1(\Omega_-)\times H^1(\Omega_+)$ such that
\begin{subequations}\label{eq:2.3}
\begin{alignat}{6}
\label{eq:2.3a}
\rho_\Sigma s^2 \bff u =&\,\Delta^* \bff u  
		& \qquad & \mbox{in $\Omega_-$},\\
\label{eq:2.3b}
(s/c)^2 v=&\,\Delta v 
		& \qquad & \mbox{in $\Omega_+$},\\
\label{eq:2.3c}
-s \gamma^-\bff u\cdot\boldsymbol\nu =&\, \partial_\nu^+ v+\lambda_0 
		 & \qquad & \mbox{in $\Gamma$},\\
\label{eq:2.3d}
\bff t^-(\bff u)=&\,-\rho_f s (\gamma^+ v +\phi_0)\boldsymbol\nu
		 & \qquad & \mbox{in $\Gamma$}.
\end{alignat}
\end{subequations}
This problem will be studied for all $s\in \mathbb C_+:=\{ s\in \mathbb C\,:\, \mathrm{Re}\,s>0\}$. The relation between \eqref{eq:2.3} and \eqref{eq:2.2} is simple: if $\lambda_0=\mathcal L\{\alpha_0\}(s)$ and $\phi_0=\mathcal L\{\beta_0\}(s)$, then the solution of \eqref{eq:2.3} is the Laplace transform of the solution of \eqref{eq:2.2}.
\paragraph{Calder\'on calculus for the acoustic problem.} The single and double layer operators associated to the Laplace resolvent equation (the Laplace transform of the wave equation) can be defined as the solution of a transmission problem. For given $(\varphi,\eta)\in H^{1/2}(\Gamma)\times H^{-1/2}(\Gamma)$ and any $s\in \mathbb C_+$, the problem of finding $v\in H^1(\mathbb R^d\setminus\Gamma)$ satisfying
\begin{alignat*}{6}
\Delta v -(s/c)^2 v =0 &\qquad & \mbox{in $\mathbb R^d\setminus\Gamma$},\\
\jump{\gamma v}=\varphi,\\
\jump{\partial_\nu v}=\eta,
\end{alignat*}
has a unique solution, which we write in terms of two linear operators, known as the single ($\mathrm S$) and double  ($\mathrm D$)  layer potentials
\[
v=\mathrm S(s/c)\eta-\mathrm D(s/c)\varphi.
\]
Associated to the potentials, we can define four integral operators
\begin{alignat*}{6}
&\mathrm V(s):=\ave{\gamma \punto}\mathrm S(s)=\gamma \mathrm S(s), & \qquad &
\mathrm K(s):=\ave{\gamma\punto} \mathrm D(s),\\
&\mathrm K^t(s):=\ave{\partial_\nu\punto}\mathrm S(s),& \qquad &
\mathrm W(s):=-\ave{\partial_\nu\punto} \mathrm D(s)=-\partial_\nu\mathrm D(s).
\end{alignat*}

\paragraph{Calder\'on calculus for the elastic problem.} In parallel to the definitions for the acoustic case, we can define the layer potentials and operators for the resolvent Navier-Lam\'e equations (Laplace transforms of the linear elasticity equations) by solving a transmission problem. For given $(\bs\varphi,\bs\eta)\in \bff H^{1/2}(\Gamma)\times \bff H^{-1/2}(\Gamma)$ and any $s\in \mathbb C_+$, we find $\bff u\in \bff H^1(\mathbb R^d\setminus\Gamma)$ satisfying
\begin{alignat*}{6}
\Delta^* \bff u -s^2 \rho_\Sigma \bff u =0 &\qquad & \mbox{in $\mathbb R^d\setminus\Gamma$},\\
\jump{\gamma \bff u}=\bs\varphi,\\
\jump{\bff t(\bff u)}=\bs\eta,
\end{alignat*}
and write the solution in terms of the data
\[
\bff u=\mathbf S(s)\bs\eta-\mathbf D(s)\bs\varphi
\]
by introducing the elastic layer potentials.
Finally, we can define four integral operators
\begin{alignat*}{6}
&\mathbf V(s):=\ave{\gamma \punto}\mathbf S(s)=\gamma\mathbf S(s), & \qquad &
\mathbf K(s):=\ave{\gamma\punto} \mathbf D(s),\\
&\mathbf K^t(s):=\ave{\bff t(\punto)}\mathbf S(s),& \qquad &
\mathbf W(s):=-\ave{\bff t(\punto)} \mathbf D(s)=-\bff t\mathbf D(s).
\end{alignat*}

\paragraph{The boundary integral system.} The boundary integral system equivalent to \eqref{eq:2.3} is derived by choosing $\bs\phi_\Sigma:=\gamma^-\bff u$ and $\phi_f:=\gamma^+ v$ as unknowns, using the representation formulas for $v$ and $\bff u$ and finally imposing the transmission conditions. The process is quite standard and we will only sketch the main steps. We introduce the matrices of operators
\[
\mathbb L(s):=
\left[\begin{array}{cc}
\bff W(s)+\rho_fs^2 \mathrm N^t\mathrm V(s/c)\mathrm N 
	& \rho_f s(\mathrm N^t\mathrm{K}(s/c)-\bff K^t(s)\mathrm N^t) \\
\rho_f s(\mathrm N\bff K(s)-\mathrm K^t(s/c)\mathrm N)
	& (\rho_f s)^2 \mathrm N\bff V(s)\mathrm N^t+\rho_f\mathrm W(s/c)
\end{array}\right]
\]
and
\[
\mathbb R(s):=
\left[\begin{array}{cc}
-\rho_f s \mathrm N^t \mathrm V(s/c) & \rho_f s(-\tfrac12\bff I+\bff K^t(s))\mathrm N^t \\
\rho_f(\tfrac12\mathrm I+\mathrm K^t(s/c)) & -(\rho_f s)^2 \mathrm N \bff V(s)\mathrm N^t
\end{array}\right].
\]
Denoting $\mathbb H^{\pm1/2}(\Gamma):=\bff H^{\pm1/2}(\Gamma)\times H^{\pm1/2}(\Gamma)$, it is easy to show that by well-known properties of the boundary integral operators on Lipschitz domains (see the general theory in \cite{McLean:2000})
$
\mathbb L(s):\mathbb H^{1/2}(\Gamma) \to 	
				\mathbb H^{-1/2}(\Gamma)
$
and
$
\mathbb R(s): H^{-1/2}(\Gamma)\times  H^{1/2}(\Gamma) \to 	
				 \mathbb H^{-1/2}(\Gamma)
$
are bounded. For the sake of notational simplicity, we will write $\mathbb L(s)(\bs\phi,\phi)$, understanding that the vector $(\bs\phi,\phi)$ is first transformed into a column vector and them left-multiplied by $\mathbb L(s)$.

\begin{theorem}
If $(\bff u,v)$ solves \eqref{eq:2.3}, then $(\bs\phi_\Sigma,\phi_f):=(\gamma^-\bff u,\gamma^+ v)$ satisfies
\begin{equation}\label{eq:2.4}
\mathbb L(s) (\bs\phi_\Sigma , \phi_f )
=\mathbb R(s) (\lambda_0 , \phi_0 ).
\end{equation}
Reciprocally, if $(\bs\phi_\Sigma,\phi_f)$ is a solution of \eqref{eq:2.4}, then
\begin{subequations}\label{eq:2.5}
\begin{align}
\label{eq:2.5a}
\bff u :=\,& -\rho_f s \bff S(s)\mathrm N^t(\phi_0+\phi_f)-\bff D(s)\bs\phi_\Sigma,\\
\label{eq:2.5b}
v :=\,& \mathrm S(s/c)(\lambda_0+s\mathrm N\bs\phi_\Sigma)+\mathrm D(s/c)\phi_f,
\end{align}
\end{subequations}
is a solution of \eqref{eq:2.3}.
\end{theorem}

\begin{proof}
If $(\bff u,v)$ satisfies \eqref{eq:2.3a}-\eqref{eq:2.3b}, then we have the representation formulas
\begin{equation}\label{eq:2.6}
\bff u=\bff S(s)\bff t^-(\bff u)-\bff D(s)\gamma^-\bff u
\quad\mbox{and}\quad
v=-\mathrm S(s/c)\partial_\nu^+ v+\mathrm D(s/c)\gamma^+ v,
\end{equation}
and the boundary integral identities
\begin{subequations}\label{eq:2.7}
\begin{alignat}{6}
\label{eq:2.7a}
&\smallfrac12\gamma^-\bff u = \bff V(s)\bff t^-(\bff u)-\bff K(s)\gamma^-\bff u,
	&\qquad & 
\smallfrac12\bff t^-(\bff u)=\bff K^t(s)\bff t^-(\bff u)+\bff W(s)\gamma^-\bff u,
\\
\label{eq:2.7b}
&\smallfrac12\gamma^+ v =-\mathrm V(s)\partial_\nu^+v+\mathrm K(s) \gamma^+ v,
	&\qquad &
\smallfrac12\partial_\nu^+ v=-\mathrm K^t(s)\partial_\nu^+ v-\mathrm W(s)\gamma^+ v.
\end{alignat}
\end{subequations}
If we define $(\bs\phi_\Sigma,\phi_f):=(\gamma^-\bff u,\gamma^+ v)$ the transmission conditions \eqref{eq:2.3c}- \eqref{eq:2.3d} become
\begin{equation}\label{eq:2.8}
\bff t^-(\bff u)=-\rho s \mathrm N^t(\phi_f+\phi_0)
\qquad \mbox{and}\qquad 
\partial_\nu^+ v=-(s\mathrm N \bs\phi_\Sigma+\lambda_0).
\end{equation} 
Substituting \eqref{eq:2.8} in \eqref{eq:2.7} and adding the two equations in \eqref{eq:2.7a} and the two in \eqref{eq:2.7b} gives the integral equations \eqref{eq:2.4}. 

Reciprocally, let $(\bff u,v)$ be defined by \eqref{eq:2.5} where $(\bs\phi_\Sigma,\phi_f)$ solve \eqref{eq:2.4}. Since $(\bff u,v)$ are defined with potentials, it follows that \eqref{eq:2.3a} and \eqref{eq:2.3b} are satisfied. Applying the limit formulas \eqref{eq:2.7} in \eqref{eq:2.5}, we see that
\begin{alignat*}{6}
 \bff t^-(\bff u)+\rho_f s\mathrm N^t(\gamma^+v+\phi_0)
 = & -\smallfrac12 \rho_f s\mathrm N^t(\phi_0+\phi_f) -\rho_f s \bff K^t(s)\mathrm N^t(\phi_0+\phi_f)+
			\bff W(s)\bs\phi_\Sigma\\
	& +\rho_f s \mathrm N^t \mathrm V(s/c)(\lambda_0+s\mathrm N \bs\phi_\Sigma)
		+\rho_f s \mathrm N^t (\smallfrac12\phi_f+\mathrm K(s/c)\phi_f) \\
=\, & (\bff W(s)+\rho_f s^2 \mathrm N^t \mathrm V(s/c)\mathrm N)\bs\phi_\Sigma +
	\rho_f s (\mathrm N^t\mathrm K(s/c)-\bff K^t(s)\mathrm N^t)\phi_f\\
	&+\rho_f s\mathrm N^t\mathrm V(s/c)\lambda_0+\rho_f s (\smallfrac12 \mathrm N^t\phi_0-\bff K^t(s)\mathrm N^t\phi_0)\\
=\, & 0,
\end{alignat*}
by the first equation in \eqref{eq:2.4}. This proves the first transmission condition \eqref{eq:2.3c}. Similarly \eqref{eq:2.3d} is proved using the second equation in \eqref{eq:2.4}.
\end{proof}

\section{Numerical discretization of the BIE system}\label{sec:3}

\subsection{Stability of Galerkin semidiscretizations in space}\label{sec:3.1}
%
\paragraph{Galerkin semidiscretization in space.} We next consider a Galerkin discretization of the integral equations \eqref{eq:2.4}. Note that when returning to the time-domain (by taking inverse Laplace transforms) this is simply a Galerkin semidiscretization in space of the system of delayed integral equations whose Laplace transform is \eqref{eq:2.4}. Following \cite{LaSa:2009a}, the study of solvability for \eqref{eq:2.4} is done at the same time as the study of Galerkin stability. We then choose two {\em closed} subspaces $\bff Y_h \subset \bff H^{1/2}(\Gamma)$ and $Y_h \subset H^{1/2}(\Gamma)$. For Galerkin semidiscretization, these spaces are taken to be finite dimensional. In the case of non-discretization (analysis of well-posedness of \eqref{eq:2.4}) we just take $\bff Y_h=\bff H^{1/2}(\Gamma)$ and $Y_h=H^{1/2}(\Gamma)$. 

The Galerkin discretization of \eqref{eq:2.4} seeks $(\bs\phi_\Sigma^h,\phi_f^h)\in \bff Y_h\times Y_h$ satisfying
\begin{equation}\label{eq:3.1}
\langle \mathbb L(s)(\bs\phi_\Sigma^h,\phi_f^h), (\bs\mu^h,\mu^h)\rangle_\Gamma=
\langle \mathbb R(s)(\lambda_0,\bs\phi_0), (\bs\mu^h,\mu^h)\rangle_\Gamma
\qquad\forall (\bs\mu^h,\mu^h)\in \bff Y_h\times Y_h.
\end{equation}
The angled bracket is the duality product of $\mathbb H^{-1/2}(\Gamma)$ with $\mathbb H^{1/2}(\Gamma)$. 
We can also write \eqref{eq:3.1} in the very condensed form
\begin{equation}\label{eq:3.2}
\mathbb L(s)(\bs\phi_\Sigma^h,\phi_f^h)-\mathbb R(s)(\lambda_0,\bs\phi_0) \in \bff Y_h^\circ\times Y_h^\circ\equiv (\bff Y_h\times Y_h)^\circ,
\end{equation}
where $X^\circ$ denotes the polar set of $X$, that is, the set of elements of the dual space that vanish on $X$. Following the same techniques of \cite{LaSa:2009a} we first rewrite the Galerkin equations \eqref{eq:3.1} as an exotic transmission problem. Note that in the new transmission problem, the elastic and the acoustic fields are defined on both sides of the boundary. 

\begin{proposition}[Transmission problem for Galerkin equations]\label{prop:3.1}
Let $(\bs\phi_\Sigma^h,\phi_f^h)\in \bff Y_h\times Y_h$ satisfy \eqref{eq:3.1} and let
\begin{subequations}\label{eq:3.3}
\begin{eqnarray}
\bff u^h &:=& -\rho_f s \bff S(s)\mathrm N^t(\phi_0+\phi_f^h)-\bff D(s)\bs\phi_\Sigma^h,\\
v^h &:=& \mathrm S(s/c)(\lambda_0+s\mathrm N\bs\phi_\Sigma^h)+\mathrm D(s/c)\phi_f^h.
\end{eqnarray}
\end{subequations}
Then $(\bff u^h,v^h)\in \bff H^1(\mathbb R^d\setminus\Gamma)\times H^1(\mathbb R^d\setminus\Gamma)$ satisfies:
\begin{subequations}\label{eq:3.4}
\begin{eqnarray}
\label{eq:3.4a}
	\Delta^* \bff u^h-\rho_\Sigma s^2 \bff u^h &=& \bff 0 \qquad \mbox{in $\mathbb R^d\setminus\Gamma$},\\
\label{eq:3.4b}
	\Delta v^h-(s/c)^2 v^h &=& 0 \qquad \mbox{in $\mathbb R^d\setminus\Gamma$},\\
\label{eq:3.4c}
	s\mathrm N\jump{\gamma\bff u^h}-\jump{\partial_\nu v^h} &=&-\lambda_0,\\
\label{eq:3.4d}
	\jump{\bff t(\bff u^h)}-\rho_f s \mathrm N^t \jump{\gamma v^h} &=& -\rho_f s \mathrm N^t\phi_0,\\
\label{eq:3.4e}
	(\jump{\gamma\bff u^h},\jump{\gamma v^h}) &\in & \bff Y_h \times Y_h,\\
\label{eq:3.4f}
	(s\mathrm N\gamma^+\bff u^h+\partial_\nu^- v^h,
		\bff t^+(\bff u^h)+\rho_f s \mathrm N^t \gamma^- v^h) &\in & \bff Y_h^\circ\times Y_h^\circ.
\end{eqnarray}
\end{subequations}
Reciprocally, given a solution of \eqref{eq:3.4} the functions
\begin{equation}\label{eq:3.5}
(\bs\phi_\Sigma^h,\phi_f^h):=(\jump{\gamma\bff u^h},-\jump{\gamma v^h}) \in \bff Y_h\times Y_h
\end{equation}
satisfy \eqref{eq:3.1}.
\end{proposition}

\begin{proof}
It is clear that the functions defined by \eqref{eq:3.3} satisfy \eqref{eq:3.4a} and \eqref{eq:3.4b}. Moreover, $\jump{\gamma \bff u^h}=\bs\phi_\Sigma^h$ and $\jump{\gamma v^h}=-\phi_f^h$, and therefore \eqref{eq:3.4e} is satisfied. At the same time
\begin{eqnarray*}
\jump{\bff t(\bff u^h)} &=&-\rho_f s\mathrm N^t(\phi_0+\phi_f^h)=-\rho_f s\mathrm N^t(\phi_0-\jump{\gamma v^h}),\\
\jump{\partial_\nu v^h} &=&\lambda_0+s\mathrm N\bs\phi_\Sigma^h=\lambda_0+s\mathrm N\jump{\gamma \bff u^h},
\end{eqnarray*}
which proves \eqref{eq:3.4c} and \eqref{eq:3.4d}. Finally, using the jump properties of the potentials, it is easy to verify that
\begin{equation}\label{eq:3.6}
(s\mathrm N\gamma^+\bff u^h+\partial_\nu^- v^h,
		\bff t^+(\bff u^h)+\rho_f s \mathrm N^t \gamma^- v^h)
	=\mathbb L(s)(\bs\phi_\Sigma^h,\phi_f^h)-\mathbb R(s)(\lambda_0,\phi_0),
\end{equation}
which proves \eqref{eq:3.4f} (see \eqref{eq:3.2}). 

Reciprocally, if we are given a solution of \eqref{eq:3.4} and we define $(\bs\phi_\Sigma^h,\phi_f^h)$ with \eqref{eq:3.5}, then by the representation formulas and \eqref{eq:3.4c}-\eqref{eq:3.4d}, it follows that we can write the fields $(\bff u^h,v^h)$ as in \eqref{eq:3.3}. We can then use \eqref{eq:3.6} again and prove that \eqref{eq:3.4f} implies \eqref{eq:3.1}.
\end{proof}

The next step consists of finding a variational formulation for \eqref{eq:3.4}. This will be done in the space
\[
\mathbb H_h:=
	\{ (\bff u^h,v^h)\in \bff H^1(\mathbb R^d\setminus\Gamma)\times H^1(\mathbb R^d\setminus\Gamma)
		\,:\,
	(\jump{\gamma\bff u^h},\jump{\gamma v^h}) \in  \bff Y_h\times Y_h\},
\]
which incorporates the only homogeneous essential transmission conditions of \eqref{eq:3.4}. 

\begin{proposition}[Equivalent variational formulation]\label{prop:3.2}
Problem \eqref{eq:3.4} is equivalent to finding
\begin{equation}\label{eq:3.7}
(\bff u^h,v^h)\in \mathbb H_h\quad\mbox{s.t.}\quad
B((\bff u^h,v^h),(\bff w,w);s)=\ell((\bff w,w);s)\qquad\forall (\bff w,w)\in \mathbb H_h,
\end{equation}
where
\begin{eqnarray*}
B((\bff u,v),(\bff w,w);s)
	&:=& (\bs\sigma(\bff u),\bs\varepsilon(\bff w))_{\mathbb R^d\setminus\Gamma}
		+\rho_\Sigma s^2 (\bff u,\bff w)_{\mathbb R^d}\\
	&& + \rho_f (\nabla v,\nabla w)_{\mathbb R^d\setminus\Gamma}+\rho_f(s/c)^2 (v,w)_{\mathbb R^d}\\
	&& +\rho_f s \big(\langle \gamma^+ v, \mathrm N\gamma^- \bff w\rangle_\Gamma 
			-\langle \gamma^- v,\mathrm N \gamma^+\bff w\rangle_\Gamma 
			\\
	&& \qquad\quad +\langle\mathrm N \gamma^+\bff u,\gamma^- w\rangle_\Gamma	
			-\langle\mathrm N\gamma^-\bff u,\gamma^+ w\rangle_\Gamma\big),\\
\ell((\bff w,w);s) &:=& \rho_f \left( \langle \lambda_0,\gamma^+ w\rangle_\Gamma 
						-s\langle \mathrm N^t\phi_0,\gamma^- \bff w\rangle_\Gamma\right).
\end{eqnarray*}
\end{proposition}

\begin{proof}
The definition of the normal traction for $\bff u$ and the normal derivative for $v$, plus simple algebraic manipulations, show that
\begin{alignat*}{6}
&\hspace{-1cm}
	\rho_f(\Delta v,w)_{\mathbb R^d\setminus\Gamma}
	+\rho_f(\nabla v,\nabla w)_{\mathbb R^d\setminus\Gamma}
	 +(\Delta^* \bff u,\bff w)_{\mathbb R^d\setminus\Gamma}
	+(\bs\sigma (\bff u),\bs\varepsilon(\bff w))_{\mathbb R^d\setminus\Gamma}\\
 =& \langle \jump{\bff t(\bff u)},\gamma^- \bff w\rangle_\Gamma
	+\langle \bff t^+(\bff u),\jump{\gamma \bff w}\rangle_\Gamma
	+\rho_f \langle \partial^- v,\jump{\gamma w}\rangle_\Gamma
	+\rho_f \langle\jump{\partial_\nu v},\gamma^+ w\rangle_\Gamma\\
=& \langle \bff t^+(\bff u)+\rho_f s \mathrm N^t\gamma^- v,\jump{\gamma\bff w}\rangle_\Gamma
	+\rho\langle\partial_\nu^- v+s\mathrm N\gamma^+\bff u,\jump{\gamma w}\rangle_\Gamma \\
	&+\langle \jump{\bff t(\bff u)}-\rho_f s \mathrm N^t\jump{\gamma v},\gamma^- \bff w\rangle_\Gamma
		+\rho_f\langle \jump{\partial_\nu v}-s\mathrm N \jump{\gamma \bff u},\gamma^+ w\rangle_\Gamma\\
	& +\rho_f s \left(\langle\mathrm N^t\gamma^- v,\gamma^+\bff w\rangle_\Gamma
			-\langle\mathrm N^t\gamma^+ v,\gamma^-\bff w\rangle_\Gamma
			+ \langle \mathrm N\gamma^-\bff u,\gamma^+ w\rangle_\Gamma
			-\langle \mathrm N\gamma^+ \bff u,\gamma^-w\rangle_\Gamma\right), 
\end{alignat*}
or equivalently
\begin{alignat}{6}\label{eq:3.8}
\nonumber
&\hspace{-2cm} B((\bff u,v),(\bff w,w);s)
		+(\Delta^* \bff u-\rho_\Sigma s^2\bff u,\bff w)_{\mathbb R^d\setminus\Gamma}
		+\rho_f(\Delta v-(s/c)^2 v,w)_{\mathbb R^d\setminus\Gamma}\\
=& \langle \bff t^+(\bff u)+\rho_f s \mathrm N^t\gamma^- v,\jump{\gamma\bff w}\rangle_\Gamma
	+\rho\langle\partial_\nu^- v+s\mathrm N\gamma^+\bff u,\jump{\gamma w}\rangle_\Gamma \\
\nonumber
	&-\rho_f\langle s\mathrm N \jump{\gamma \bff u}-\jump{\partial_\nu v},\gamma^+ w\rangle_\Gamma
		+\langle \jump{\bff t(\bff u)}-\rho_f s \mathrm N^t\jump{\gamma v},\gamma^- \bff w\rangle_\Gamma.
\end{alignat}
From here it is clear that a solution of \eqref{eq:3.4} satisfies \eqref{eq:3.7}. Reciprocally, if we have a solution of \eqref{eq:3.7}, 	testing with smooth functions with compact support in $\mathbb R^d\setminus\Gamma$, we can easily recover equations \eqref{eq:3.4a} and \eqref{eq:3.4b}. Therefore, by \eqref{eq:3.8} it follows that
\begin{alignat*}{6}
-\rho_f\langle s\mathrm N \jump{\gamma \bff u^h}-\jump{\partial_\nu v^h}+\lambda_0,\gamma^+ w\rangle_\Gamma
		+\langle \jump{\bff t(\bff u^h)}-\rho_f s \mathrm N^t(\jump{\gamma v^h}+\phi_0),\gamma^- \bff w\rangle_\Gamma &\\
+\langle \bff t^+(\bff u^h)+\rho_f s \mathrm N^t\gamma^- v^h,\jump{\gamma\bff w}\rangle_\Gamma
	+\rho\langle\partial_\nu^- v^h+s\mathrm N\gamma^+\bff u^h,\jump{\gamma w}\rangle_\Gamma 
		&=0,
\end{alignat*}
for all $(\bff w,w)\in \mathbb H_h$. The transmission conditions \eqref{eq:3.4c}, \eqref{eq:3.4d}, and \eqref{eq:3.4f} follow from the simple observation that the map
\[
\begin{array}{rcl} 
	\mathbb H_h & \longrightarrow &H^{1/2}(\Gamma) \times \bff H^{1/2}(\Gamma)\times \bff Y_h \times Y_h \\
	(\bff w,w) & \longmapsto & (\gamma^+ w,\gamma^-\bff w,\jump{\gamma\bff w},\jump{\gamma w})
\end{array}
\]
is surjective.
\end{proof}

The third step in the analysis is the proof of well-posedness of the variational problem \eqref{eq:3.7}. Following \cite{LaSa:2009a}, we define the energy norm
\[
\triple{(\bff u,v)}_{|s|}^2:=
	(\bs\sigma(\bff u),\bs\varepsilon(\overline{\bff u}))_{\mathbb R^d\setminus\Gamma}
	+  \| s \sqrt{\rho_\Sigma} \bff u\|_{\mathbb R^d}^2 
	+\rho_f \| \nabla v\|_{\mathbb R^d\setminus\Gamma}^2
	+\rho_f c^{-2} \| s\,v\|_{\mathbb R^d}^2.
\]
We will also write $\sigma:=\mathrm{Re}\,s>0$ (for all $s\in \mathbb C_+$) and $\underline\sigma:=\min\{\sigma,1\}$. To shorten some of the forthcoming expressions, we will denote:
\begin{alignat*}{6}
\| (\bff u,v)\|_{1,\mathbb R^d\setminus\Gamma}^2 
	&:=(\boldsymbol\sigma(\mathbf u),\boldsymbol\varepsilon(\mathbf u))_{\mathbb R^d\setminus\Gamma}^2
		+\| \sqrt{\rho_\Sigma}\bff u\|_{\mathbb R^d}^2
		+\rho_f\|\nabla v\|_{\mathbb R^d\setminus\Gamma}^2
		+\rho_f c^{-2}\| v\|_{\mathbb R^d}^2,\\
\|(\bs\phi,\phi)\|_{1/2,\Gamma}^2 
	&:=\|\bs\phi\|_{1/2,\Gamma}^2+\|\phi\|_{1/2,\Gamma}^2,\\
\|(\lambda,\varphi)\|_{-1/2,1/2,\Gamma}^2
	&:=\|\lambda\|_{-1/2,\Gamma}^2+\|\varphi\|_{1/2,\Gamma}^2.
\end{alignat*}
Note that the energy norm and the first of the above norms are related by
\begin{equation}\label{eq:3.9}
\underline\sigma \| (\bff u,v)\|_{1,\mathbb R^d\setminus\Gamma}
	\le \triple{(\bff u,v)}_{|s|}
	\le \frac{|s|}{\underline\sigma}\| (\bff u,v)\|_{1,\mathbb R^d\setminus\Gamma}.
\end{equation}
Finally, the expression {\em independent of $h$} will be used to mean independent of the choice of the spaces $\bff Y_h$ and $Y_h$.

\begin{proposition}[Well-posedness]\label{prop:3.3}
Problem \eqref{eq:3.7} is uniquely solvable for any $(\lambda_0,\phi_0)\in H^{-1/2}(\Gamma)\times H^{1/2}(\Gamma)$ and $s\in \mathbb C_+$. Moreover, there exists $C>0$, independent of $h$, such that
\begin{equation}\label{eq:3.10}
\triple{(\bff u^h,v^h)}_{|s|} \le C \frac{|s|}{\underline\sigma\, \sigma} 
	\|(\lambda_0,s\phi_0)\|_{-1/2,1/2,\Gamma}.
\end{equation}
\end{proposition}

\begin{proof}
A simple computation shows that
\begin{equation}\label{eq:3.11}
\mathrm{Re}\,\left(\overline s B((\bff u,v),(\overline{\bff u},\overline v);s)\right)=
\sigma \triple{(\bff u,v)}_{|s|}^2,
\end{equation}
and that
\begin{equation}\label{eq:3.12}
|\ell ((\bff u,v);s)| \le C \rho_f \| (\lambda_0,s\phi_0)\|_{-1/2,1/2,\Gamma} \|(\bff u,v)\|_{1,\mathbb R^d\setminus\Gamma},
\end{equation}
which proves well-posedness of \eqref{eq:3.7} by the Lax-Milgram lemma. The estimate \eqref{eq:3.10} is a direct consequence of \eqref{eq:3.11} and \eqref{eq:3.12}, using \eqref{eq:3.9} to relate the norms. 
\end{proof}

The final step wraps up the analysis by collecting information from the previous results.

\begin{corollary}\label{cor:3.4}
Equations \eqref{eq:3.1} are uniquely solvable for all $s\in\mathbb C_+$ and any choice of the closed spaces $\bff Y_h$ and $Y_h$. Moreover, if $(\bff u^h,v^h)$ are defined using \eqref{eq:3.3} from the solution of \eqref{eq:3.1}, the following bounds hold with $C>0$ independent of $h$:
\begin{subequations}\label{eq:3.13}
\begin{eqnarray}
\label{eq:3.13a}
\| (\bff u^h,v^h)\|_{1,\mathbb R^d\setminus\Gamma}
	&\le & C \frac{|s|}{\underline\sigma^2 \sigma} \| (\lambda_0,s\phi_0)\|_{-1/2,1/2,\Gamma},\\
\label{eq:3.13b}
\| (\bs\phi^h_\Sigma,\phi^h_f)\|_{1/2,\Gamma}
	&\le & C \frac{|s|}{\underline\sigma^2 \sigma} \| (\lambda_0,s\phi_0)\|_{-1/2,1/2,\Gamma}.
\end{eqnarray}
\end{subequations}
\end{corollary}

\begin{proof}
Propositions \ref{prop:3.1} and \ref{prop:3.2} relate the discrete integral system \eqref{eq:3.1} to the variational problem \eqref{eq:3.7}, which is shown to be uniquely solvable in Proposition \ref{prop:3.3}. The estimate \eqref{eq:3.13a} follows from \eqref{eq:3.10} and \eqref{eq:3.9}. Finally, the bound \eqref{eq:3.13b} follows from \eqref{eq:3.13a} and \eqref{eq:3.5}.
\end{proof}

We end this section by noting that Corollary \ref{cor:3.4} implies the unique solvability of the semidiscrete equations that are obtained by taking the inverse Laplace transform of \eqref{eq:3.1}. They can also be translated into a time-domain estimate that bounds norms of the solution in terms of bounds for the data.

%
\subsection{The effect of Galerkin semidiscretization}
%

In this section we analyze the effect of space semidiscretization, that is, we estimate the difference between the solution of \eqref{eq:2.4} and \eqref{eq:3.1}. The analysis follows a very similar pattern to the one displayed in Section \ref{sec:3.1}. We start by writing the error equations:
\begin{equation}\label{eq:3.14}
\mathbb L(s) (\bs\phi^h_\Sigma-\bs\phi_\Sigma, \phi^h_f-\phi_f)\in \bff Y_h^\circ\times Y_h^\circ.
\end{equation}
We will develop the analysis in terms of the variables
$\bff e^h:=\bff u^h-u$ and $e^h:=v^h-v,$
from which the error of the boundary unknowns can be recovered:
\begin{equation}\label{eq:3.15}
(\bs\phi^h_\Sigma-\bs\phi_\Sigma, \phi^h_f-\phi_f)=( \jump{\gamma \bff e^h},-\jump{\gamma e^h}).
\end{equation}
The potential representation for $(\bff e^h,e^h)$ is obtained by subtracting \eqref{eq:2.5} from \eqref{eq:3.3}
\begin{subequations}\label{eq:3.16}
\begin{eqnarray}
\bff e^h &=& -\rho_f s \bff S(s)\mathrm N^t(\phi_f^h-\phi_f)-\bff D(s)(\bs\phi_\Sigma^h-\bs\phi_\Sigma),\\
e^h &=& s\mathrm S(s/c)\mathrm N(\bs\phi_\Sigma^h-\bs\phi_\Sigma)+\mathrm D(s/c)(\phi_f^h-\phi_f).
\end{eqnarray}
\end{subequations}
The proofs of the following results are quite similar to those of Propositions \ref{prop:3.1}, \ref{prop:3.2}, and \ref{prop:3.3}. We will only point out the main differences.

\begin{proposition}\label{prop:3.5}
The error potentials $\bff e^h:=\bff u^h-u$ and $e^h:=v^h-v$ satisfy:
\begin{subequations}\label{eq:3.17}
\begin{eqnarray}
\label{eq:3.17a}
	\Delta^* \bff e^h-\rho_\Sigma s^2 \bff e^h &=& \bff 0 \qquad \mbox{in $\mathbb R^d\setminus\Gamma$},\\
\label{eq:3.17b}
	\Delta e^h-(s/c)^2 e^h &=& 0 \qquad \mbox{in $\mathbb R^d\setminus\Gamma$},\\
\label{eq:3.17c}
	s\mathrm N\jump{\gamma\bff e^h}-\jump{\partial_\nu e^h} &=&0,\\
\label{eq:3.17d}
	\jump{\bff t(\bff e^h)}-\rho_f s \mathrm N^t \jump{\gamma e^h} &=&0,\\
\label{eq:3.17e}
	(\jump{\gamma\bff e^h},\jump{\gamma e^h})+(\bs\phi_\Sigma,-\phi_f) &\in & \bff Y_h \times Y_h,\\
\label{eq:3.17f}
	(s\mathrm N\gamma^+\bff e^h+\partial_\nu^- e^h,
		\bff t^+(\bff e^h)+\rho_f s \mathrm N^t \gamma^- e^h) &\in & \bff Y_h^\circ\times Y_h^\circ.
\end{eqnarray}
\end{subequations}
Reciprocally, given a solution of \eqref{eq:3.17}, the quantities defined in \eqref{eq:3.16} satisfy \eqref{eq:3.14}
\end{proposition}

\begin{proposition}\label{prop:3.6}
Problem \eqref{eq:3.17} is equivalent to the variational problem: find $(\bff e^h,e^h)\in \bff H^1(\mathbb R^d\setminus\Gamma)\times H^1(\mathbb R^d\setminus\Gamma)$ such that
\begin{subequations}\label{eq:3.18}
\begin{alignat}{6}
\label{eq:3.18a}
(\jump{\gamma\bff e^h}+\bs\phi_\Sigma,\jump{\gamma e^h}-\phi_f) \in \,& \bff Y_h \times Y_h,&&\\
\label{eq:3.18b}
B((\bff e^h,e^h),(\bff w,w);s)= \,& 0\qquad \qquad&&\forall (\bff w,w)\in \mathbb H_h,
\end{alignat}
\end{subequations}
\end{proposition}

We note that, in comparison with \eqref{eq:3.7}, problem \eqref{eq:3.18} has homogeneous right-hand side but incorporates a side restriction
\eqref{eq:3.18a}. This compares with how the conditions \eqref{eq:3.4c}-\eqref{eq:3.4d} have become homogeneous in \eqref{eq:3.17c}-\eqref{eq:3.17d}, while the homogeneous condition \eqref{eq:3.4e} is now non-homogeneous \eqref{eq:3.17e}.

\begin{proposition}\label{prop:3.7}
Problem \eqref{eq:3.18} is uniquely solvable for any $(\bs\phi_\Sigma,\phi_f)\in \mathbb H^{1/2}(\Gamma)$ and $s\in \mathbb C_+$. Moreover, there exists $C>0$ independent of $h$ such that
\begin{equation}\label{eq:3.19}
\triple{(\bff e^h,e^h)}_{|s|}
\le C \frac{|s|^2}{\sigma\underline\sigma}\|(\bs\phi_\Sigma,\phi_f)\|_{1/2,\Gamma}.
\end{equation}
\end{proposition}

\begin{proof}
Using the definition of the bilinear form $B$ (see Proposition \ref{prop:3.2}) and \eqref{eq:3.9}, we can easily bound
\begin{align}
\nonumber
|B((\bff u,v),(\bff w,w);s)| 
	\le \, & 
	\triple{(\bff u,v)}_{|s|}\triple{(\bff w,w)}_{|s|}
	+C |s| \, \|(\bff u,v)\|_{1,\mathbb R^d\setminus\Gamma}\|(\bff w,w)\|_{1,\mathbb R^d\setminus\Gamma}\\
\label{eq:3.20}
	\le\, & C \frac{|s|}{\underline\sigma} 
		\|(\bff u,v)\|_{1,\mathbb R^d\setminus\Gamma}\triple{(\bff w,w)}_{|s|}.
\end{align}
Take now $(\bff w,w)\in \bff H^1(\mathbb R^d\setminus\Gamma)\times H^1(\mathbb R^d\setminus\Gamma)$  such that
\begin{equation}\label{eq:3.21}
\jump{\gamma\bff w}=\bs\phi_\Sigma,\quad
\jump{\gamma w}=-\phi_f,\quad
\|(\bff w,w)\|_{1,\mathbb R^d\setminus\Gamma}\le C \|(\bs\phi_\Sigma,\phi_f)\|_{1/2,\Gamma}.
\end{equation}
By the energy identity \eqref{eq:3.11},  the fact that $(\bff e^h+\bff w,e^h+w)\in \mathbb H_h$, and \eqref{eq:3.20}, it follows that
\begin{eqnarray*}
\triple{(\bff e^h+\bff w,e^h+w)}^2_{|s|}
	&\le & \frac{|s|}{\sigma} |B((\bff e^h+\bff w,e^h+w),(\bff e^h+\bff w,e^h+w);s)|\\
	&= & \frac{|s|}{\sigma} |B((\bff w,w),(\bff e^h+\bff w,e^h+w);s)|\\
	&\le & C \frac{|s|^2}{\sigma\underline\sigma} \|(\bff w,w)\|_{1,\mathbb R^d\setminus\Gamma}
			\triple{(\bff e^h+\bff w,e^h+w)}_{|s|}.
\end{eqnarray*}
Therefore, using \eqref{eq:3.9}
\[
\triple{(\bff e^h,e^h)}_{|s|} \le C
 \frac{|s|^2}{\sigma\underline\sigma} \|(\bff w,w)\|_{1,\mathbb R^d\setminus\Gamma},
\]
and the result follows from \eqref{eq:3.21}. For readers who are acquainted with this kind of Laplace-domain estimates, let us clarify that the use of the optimal $|s|$-dependent lifting of Bamberger-HaDuong \cite[Lemma 1]{BaHa:1986a}  (see also \cite[Proposition 2.5.1]{Sayas:2014})  instead of the plain lifting used in \eqref{eq:3.21} does not improve the estimate. This is principally due to the $s$ factor in the boundary terms of the bilinear form $B$.
\end{proof}

\begin{corollary}\label{cor:3.8}
Let $(\bs\phi_\Sigma,\phi_f)$ and $(\bs\phi^h_\Sigma,\phi^h_f)$ be the respective solutions of \eqref{eq:2.4} and \eqref{eq:3.1}. Let then $(\bff u,v)$ and $(\bff u^h,v^h)$ be defined through \eqref{eq:2.5} and \eqref{eq:3.3} respectively. Then there exists $C>0$ independent of $h$ such that
\begin{eqnarray*}
\| (\bff u^h-\bff u,v^h-v)\|_{1,\mathbb R^d\setminus\Gamma}
	&\le & C \frac{|s|^2}{\sigma\underline\sigma^2} \| (\bs\phi_\Sigma,\phi_f)\|_{1/2,\Gamma},\\
\| (\bs\phi^h_\Sigma-\bs\phi_\Sigma,\phi^h_f-\phi_f)\|_{1/2,\Gamma}
	&\le & C \frac{|s|^2}{\sigma\underline\sigma^2} \| (\bs\phi_\Sigma,\phi_f)\|_{1/2,\Gamma}.
\end{eqnarray*}
\end{corollary}

\begin{proof}
The result is a direct consequence of Propositions \ref{prop:3.5}, \ref{prop:3.6}, and \ref{prop:3.7}.
\end{proof}
Using the results obtained in the previous two subsections it is possible to establish error estimates in the time-domain. Data will be taken in the Sobolev spaces
\[
W^k_+(H^{\pm1/2}(\Gamma)):=\{ \xi\in \mathcal C^{k-1}(\mathbb R;H^{\pm1/2}(\Gamma))\,:\, \xi \equiv 0 \mbox{ in $(-\infty,0)$}, \xi^{(k)}\in L^1(\mathbb R;H^{\pm1/2}(\Gamma))\},
\]
for $k\ge 1$.
A straightforward application of the inversion theorem of the Laplace transform \cite[Theorem 7.1]{DoSa:2013} (see also \cite[Proposition 3.2.2]{Sayas:2014}) starting with the bounds of Corollary \ref{cor:3.4} yields the following:
\\
\begin{corollary}\label{cor:3.9}
If the data of the problem satisfy $\lambda_0\in W^3_+(H^{-1/2}(\Gamma))$, $\phi_0\in W^4_+(H^{1/2}(\Gamma))$, then
 $(\boldsymbol\phi_{\Sigma},\phi_{f})$ and $(\mathbf{u}^h,v^h)$ are continuous causal functions of time and  for all $t\ge 0$
\[
\|(\boldsymbol\phi_{\Sigma},\phi_{f})(t)\|_{1/2,\Gamma}
\leq \frac{D_1t^2}{t+1}\max\{1,t^2\}\int_0^t\|\mathcal{P}_{3}(\lambda_0,\dot\phi_0)(\tau)\|_{-1/2,1/2, \Gamma}\;d\tau, 
\]
\[
\|(\mathbf{u}^h,v^h)(t)\|_{1,\mathbb{R}^d\setminus\Gamma} 
\leq \frac{D_2t^2}{t+1}\max\{1,t^2\}\int_0^t\|\mathcal{P}_{3}(\lambda_0,\dot\phi_0)(\tau)\|_{-1/2,1/2,\Gamma}\;d\tau, 
\]
where $D_1$ and $D_2$ depend only on $\Gamma$ and 
\[
(\mathcal{P}_{k}f)(t) := \displaystyle\sum_{l=0}^{k} {k\choose l} f^{(l)}(t).
\]
\end{corollary}
In a similar fashion, a combined application of \cite[Theorem 7.1]{DoSa:2013}  and Corollary \ref{cor:3.8}, provides the following estimate for the errors of semidiscretization in time. Note that we are allowed to insert the best approximation operators in the right-hand side of the bound of Corollary \ref{cor:3.10} because the error produced by trying to compute the exact solution and the difference of the exact solution with its best approximation is the same.

\begin{corollary}\label{cor:3.10}
If the exact solution of \eqref{eq:2.4} satisfies $\boldsymbol\phi_\Sigma\in W^4_+(\mathbf H^{1/2}(\Gamma))$ and $\phi_f\in W^4_+(H^{1/2}(\Gamma))$, then $(\mathbf{e}^h,e^h):=(\mathbf{u}-\mathbf{u}^h,v-v^h)\in\mathcal{C}(\mathbb{R},\mathbf{H}^1(\mathbb{R}^d\setminus\Gamma)\times H^1(\mathbb{R}^d\setminus\Gamma))$  and for all $t\ge 0$ we have the bound
\[
\|(\mathbf{e}^h,e^h)(t)\|_{1,\mathbb{R}^d\setminus\Gamma} \leq \frac{Dt^2}{t+1}\max\{1,t^2\}\int_0^t\|\mathcal{P}_{4}(\boldsymbol\phi_{\Sigma}-\boldsymbol\Pi_h\boldsymbol\phi_{\Sigma},\phi_f -\Pi_h\phi_f)(\tau)\|_{1/2,\Gamma}\;d\tau,
\]
where $\boldsymbol\Pi_h$ and $\Pi_h$ are the best approximation operators in $\mathbf Y_h$ and $Y_h$, and $D$ depends only on $\Gamma$.
\end{corollary}

\subsection{A fully discrete method}

\paragraph{Full discretization with BDF2-CQ.}
A fully discrete method can be obtained by using any of the many Convolution Quadrature schemes. The reader is referred to \cite{Lubich:1994, BaSc:2012,HaSa:2014} for the algorithmic description of multistep and multistage CQ schemes. We next give an estimate for the BDF2-based CQ method, based on the stability bound in the Laplace-domain obtained in Proposition \ref{prop:3.3} and \cite[Proposition 4.6.1]{Sayas:2014} (a slight refinement of one of the main convergence theorems in \cite{Lubich:1994}).

\begin{proposition}\label{prop:3.11}
Let $\ell=6$ and $(\lambda_0,\phi_0)$ be causal problem data such that
$\lambda_0\in W^\ell_+(H^{-1/2}(\Gamma))$ and $\phi_0\in W^{\ell+1}_+(H^{1/2}(\Gamma))$. 
Then
\[
\|(\mathbf u^h,v^h)(t)-(\mathbf u^h_\kappa,v^h_\kappa)(t)\|_{1,\mathbb R^d\setminus\Gamma} \leq D\kappa^2(1+t^2)\int_0^t\|(\lambda_0^{(\ell)},\phi_0^{(\ell+1)})(\tau)\|_{-1/2,1/2,\Gamma}\,d\tau.
\]
\end{proposition}
It is important to note that the high-order regularity $\ell=6$ is only required to achieve optimal convergence of order $\kappa^2$. For problem data with regularity as low as $\ell=3$, reduced convergence of order $\kappa^{3/2}$ is achieved (see \cite{Sayas:2014}).

%
\section{General linear elastic materials: BEM-FEM}
%

Going back to the system of equations \eqref{eq:2.3}, an alternate approach aiming for a finite element solution of the elastic wavefield and a boundary element solution of the acoustic wavefield is to use a direct boundary integral representation of the acoustic wave while keeping the partial differential equation for the elastic displacement in variational form. This approach is particularly well suited for the case of variable elastic density and Lam\'e coefficients, and also for heterogeneous anisotropic materials. In the following we assume that the stress is given by a linear law $\boldsymbol \sigma=\mathbf C(\mathbf x)\,\boldsymbol\varepsilon$, where for each $\mathbf x\in \Omega_-$, $\mathbf C(\mathbf x)$ is a linear operator that transforms symmetric matrices into symmetric matrices and satisfies
$
\boldsymbol\varepsilon:\mathbf C(\mathbf x)\boldsymbol\varepsilon \ge C_0 \boldsymbol\varepsilon:\boldsymbol\varepsilon 
$
for some positive constant $C_0$, for every symmetric matrix $\boldsymbol\varepsilon$ and for almost every $\mathbf x\in \Omega_-$. We also assume that the components of the tensor $\mathbf C$ are $L^\infty(\Omega_-)$ functions and that the solid density $\rho_\Sigma\in L^\infty(\Omega_-)$ is strictly positive. The Navier-Lam\'e operator is now given by $\Delta^*\mathbf u=\nabla\cdot (\mathbf C \boldsymbol\varepsilon(\mathbf u))$ and the traction operator $\mathbf t$ is redefined accordingly as well.

The derivation employs standard arguments of boundary integral equations and is presented with careful detail in \cite{HsSaWe:2015}, with the resulting equivalent system being
\begin{subequations}\label{eq:4.1}
\begin{alignat}{6}
\label{eq:4.1a}
\rho_\Sigma s^2 \bff u -\Delta^* \bff u=&\,0  
		& \qquad & \mbox{in $\Omega_-$},\\
\label{eq:4.1b}
\bff t^-(\bff u)+\rho_f s \mathrm N^t\phi=&\,-\rho_f s \mathrm N^t\phi_0
		 & \qquad & \mbox{on $\Gamma$},\\
\label{eq:4.1c}
\mathrm{V}(s/c)\lambda + \left(\tfrac{1}{2}\mathrm{I}-\mathrm{K}(s/c)\right)\phi =&\, 0
		& \qquad & \mbox{on $\Gamma$},\\
\label{eq:4.1d}
\left(-\tfrac{1}{2}\mathrm{I}+\mathrm{K}^{t}(s/c)\right)\lambda + \mathrm{W}(s/c)\phi -s\mathrm{N}\gamma\mathbf{u} =&\, \lambda_0 
		 & \qquad & \mbox{on $\Gamma$}.
\end{alignat}
\end{subequations}
For notational convenience, we introduce the interior  elastodynamic bilinear form in the Laplace-domain
\[
a(\mathbf u,\mathbf w;s):=
	(\boldsymbol\sigma(\mathbf u),\boldsymbol\varepsilon(\mathbf w))_{\Omega_-}
	+s^2(\rho_f\mathbf u,\mathbf w)_{\Omega_-},
\]
so that the variational formulation of \eqref{eq:4.1a}-\eqref{eq:4.1b} reads
\[
a(\mathbf u,\mathbf w;s)+s\langle\rho_f (\phi+\phi_0),\gamma\mathbf w\cdot\boldsymbol\nu\rangle_\Gamma=0
\quad\forall \mathbf w\in \mathbf H^1(\Omega_-).
\]
We note that the operator $\mathrm N \gamma\mathbf w=\gamma \mathbf w\cdot\boldsymbol\nu$ appears in this weak formulation, while $\mathrm N^t$ will not be used any longer in this section. Since the language of this section is less heavy on the side of operators, we will keep the explicit form of the combined operator $\mathrm N\gamma$ as a trace operator dotted with the normal vector field.

%
\subsection{Galerkin semidiscretization in space}
%
Just as in Section \ref{sec:3.1}, the solvability and stablity of \eqref{eq:4.1} are studied simultaneously. In order to do so, we define the  closed subspaces 
\[
\mathbf{V}_h\subset \mathbf{H}^1(\Omega_-), \quad X_h\subset H^{-1/2}(\Gamma),\quad Y_h\subset H^{1/2}(\Gamma).
\]
The following result establishes the connection between the discrete counterpart of problem \eqref{eq:4.1} and a non-standard transmission problem. Note that the `Finite Element' form is a discretization of the interior Navier-Lam\'e equation, and therefore, the elastic operator has been discretized, as opposed to what happens with the `Boundary Element' counterpart, where only transmission conditions are discretized.

\begin{proposition}[Transmission problem for Galerkin equations]\label{prop:4.1}
If $(\mathbf{u}^h,\, \phi^h,\, \lambda^h)\,\in\,\mathbf{V}_h\times Y_h\times X_h$ satisfies the Galerkin equations
\begin{subequations}\label{eq:4.2}
\begin{alignat}{6}
\label{eq:4.2a}
a(\mathbf u^h,\mathbf w;s)+s\langle\rho_{f}(\phi^h+\phi_0),\gamma\mathbf{w}\cdot\boldsymbol\nu\rangle_{\Gamma} &=0
                             & \qquad & \forall \mathbf{w}\in\mathbf{V}_h,\\
\label{eq:4.2b}
-s\gamma\mathbf{u}^h\cdot\boldsymbol\nu + \mathrm{W}(s/c)\phi^h + \left(-\tfrac{1}{2}\mathrm{I}+\mathrm{K}^{t}(s/c)\right)\lambda^h - \lambda_0 & \,\in\, Y_h^{\circ},\\
\label{eq:4.2c}
\left(\tfrac{1}{2}\mathrm{I}-\mathrm{K}(s/c)\right)\phi^h + \mathrm{V}(s/c)\lambda^h &  \,\in\,X_h^{\circ},
\end{alignat}
\end{subequations}
and
\begin{equation}\label{eq:4.3}
v^h := \mathrm{D}(s/c)\phi^h - \mathrm{S}(s/c)\lambda^h,
\end{equation}
then the pair $(\mathbf{u}^h, v^h)\,\in\,\mathbf{V}_h\times H^1(\mathbb R^d\setminus\Gamma)$ satisfies the transmission problem
\begin{subequations}\label{eq:4.4}
\begin{alignat}{6}
\label{eq:4.4a}
a(\mathbf u^h,\mathbf w;s)+s\langle\rho_{f}(-\jump{\gamma v^h}+\phi_0),\gamma\mathbf{w}\cdot\boldsymbol\nu\rangle_{\Gamma} &=0
                             & \qquad & \forall \mathbf{w}\in\mathbf{V}_h,\\\label{eq:4.4b}
-\Delta v^h +(s/c)^2v^h =&\,0 & \qquad & \mbox{in\; $\mathbb{R}^d\setminus\Gamma$},\\
\label{eq:4.4c}
\jump{\gamma v^h} \,\in &\, Y_h,\\
\label{eq:4.4d}
\jump{\partial_{\nu}v^h} \,\in &\,X_h,\\
\label{eq:4.4e}
s\gamma\mathbf{u}^h \cdot\boldsymbol\nu+ 
	\partial_\nu^{+}v^h+\lambda_0 \,\in &\, Y_h^{\circ},\\
\label{eq:4.4f}
\gamma^-v^h \,\in &\, X_h^{\circ}.
\end{alignat}
\end{subequations}
Conversely, given a solution of \eqref{eq:4.4}, the triplet
\begin{equation}\label{eq:4.5}
(\mathbf{u}^h,\phi^h, \lambda^h) := (\mathbf{u}^h,-\jump{\gamma v^h}, -\jump{\partial_{\nu}v^h}) \,\in\, \mathbf{V}_h\times Y_h\times X_h 
\end{equation}
satisfies \eqref{eq:4.2}.
\end{proposition} 
\begin{proof}
Equations \eqref{eq:4.4b}, \eqref{eq:4.4c}, and \eqref{eq:4.4d} are simple consequences of the definition of $v^h$ and the jump relations of the double and single layer potentials. Moreover, using the definition of $v^h$ and the well known identities
\[
\partial_{\nu}^-\mathrm S(s) = \tfrac{1}{2}\mathrm I +\mathrm{K}^t(s) \,,\qquad \gamma^-\mathrm D (s) = -\tfrac{1}{2}\mathrm I + \mathrm{K}(s),
\]
it is easy to verify that \eqref{eq:4.4e} and \eqref{eq:4.4f} are just restatements of \eqref{eq:4.2b} and \eqref{eq:4.2c}.

To prove the converse, note that \eqref{eq:4.4b} and the definition of $(\phi^h,\lambda^h)$ in \eqref{eq:4.5} imply the integral representation \eqref{eq:4.3}. Then \eqref{eq:4.2b} is equivalent to \eqref{eq:4.4e} and \eqref{eq:4.2c} is equivalent to \eqref{eq:4.4f}.
\end{proof}

\begin{proposition}[Equivalent variational formulation]\label{prop:4.2}
Let
\[
V_h := \{v\in H^1(\mathbb R^d\setminus\Gamma): \jump{\gamma v}\in Y_h\,,\,\gamma^-v\in X_h^{\circ}\}.
\]
The problem \eqref{eq:4.4} is equivalent to finding $(\mathbf{u}^h,v^h)\in \mathbf{V}_h\times V_h$ such that
\begin{equation}\label{eq:4.6}
\mathcal{A}\left( (\mathbf{u}^h,v^h),(\mathbf{w},w);s\right) = f\left((\mathbf{w},w);s\right) \quad \forall(\mathbf{w},w)\in \mathbf{V}_h\times V_h,
\end{equation}
where
\begin{align*}
\mathcal{A}\left((\mathbf{u},v),(\mathbf{w},w);s\right):=& \left(\boldsymbol\sigma(\mathbf{u}),\boldsymbol\varepsilon(\mathbf{w})\right)_{\Omega_-} + s^2\left(\rho_{\Sigma}\mathbf{u},\mathbf{w}\right)_{\Omega_-} \\
		& +\rho_f\left(\nabla v,\nabla w\right)_{\mathbb{R}^d\setminus\Gamma} + \rho_f(s/c)^2\left(v,w\right)_{\mathbb{R}^d\setminus\Gamma}\\
		& + \rho_fs\langle\gamma \mathbf{u}\cdot\boldsymbol\nu,\jump{\gamma w} \rangle_{\Gamma} - \rho_fs\langle\jump{\gamma v},\gamma \mathbf{w}\cdot\boldsymbol\nu \rangle_{\Gamma},
\end{align*}
and
\[
f\left((\mathbf{w},w);s\right):= - \rho_fs\langle\phi_0,\gamma \mathbf{w}\cdot\boldsymbol\nu \rangle_{\Gamma}- \rho_f\langle\lambda_0,\jump{\gamma w} \rangle_{\Gamma}.
\]
\end{proposition}
\begin{proof}
Let $(\mathbf u^h,v^h)$ be a solution pair for \eqref{eq:4.4}. Then, for all $w\in V_h$,
\begin{eqnarray*}
\langle \partial_\nu^+ v^h,\jump{\gamma w}\rangle_\Gamma
	&=& \langle \partial_\nu^- v^h,\gamma^- w\rangle_\Gamma
		- \langle \partial_\nu^+ v^h,\gamma^+ w\rangle_\Gamma
		-\langle\jump{\partial_\nu v^h},\gamma^- w\rangle_\Gamma \\
	&=& (\nabla v^h,\nabla w)_{\mathbb R^d\setminus\Gamma}
		+(s/c)^2 (v^h,w)_{\mathbb R^d},
\end{eqnarray*}
after applying \eqref{eq:4.4b} and \eqref{eq:4.4f}. Therefore, testing \eqref{eq:4.4e} with $\jump{\gamma w}$ for $w\in V_h$, and substituting the above, it follows that
\begin{equation}\label{eq:4.7}
(s/c)^2(v^h,w)_{\mathbb{R}^d}+(\nabla v^h,\nabla w)_{\mathbb{R}^d\setminus\Gamma}+s\langle\gamma \mathbf{u}^h\cdot\boldsymbol\nu,\jump{\gamma w}\rangle_{\Gamma} = -\langle\lambda_0,\jump{\gamma w}\rangle_{\Gamma}
\quad \forall w\in V_h. 
\end{equation}
However, the pair of equations \eqref{eq:4.4a} and \eqref{eq:4.7} are equivalent to \eqref{eq:4.6}.

To prove the converse statement, note that we need to show that a solution of \eqref{eq:4.7} satisfies \eqref{eq:4.4b}, \eqref{eq:4.4d}, and \eqref{eq:4.4e}. Equation \eqref{eq:4.7} applied to a general compactly supported $w\in \mathcal C^\infty(\mathbb R^d\setminus\Gamma)$ is the distributional form of \eqref{eq:4.4b}. Therefore, \eqref{eq:4.7} (after integration by parts) implies
\[
\langle \partial_\nu^- v^h,\gamma^- w\rangle_\Gamma
- \langle \partial_\nu^+ v^h,\gamma^+ w\rangle_\Gamma
+\langle s\gamma\mathbf{u}^h\cdot\boldsymbol\nu+\lambda_0,\jump{\gamma w}\rangle_{\Gamma}=0 \quad\forall w\in V_h,
\]
which, after some simple algebra, is shown to be equivalent to
\begin{equation}\label{eq:4.8}
\langle \partial_{\nu}^+v^h+s\gamma\mathbf{u}^h\cdot\boldsymbol\nu+\lambda_0,\jump{\gamma w}\rangle_{\Gamma} +\langle \jump{\partial_{\nu}v^h},\gamma^-w\rangle_{\Gamma} = 0
\qquad\forall w\in V_h.
\end{equation}
However, the operator $V_h\ni w \longmapsto (\jump{\gamma w},\gamma^-w)\in  Y_h\times X^{\circ}_h$ is surjective, and therefore \eqref{eq:4.8} is equivalent to \eqref{eq:4.4d} and \eqref{eq:4.4e}.
\end{proof}

For the analysis of \eqref{eq:4.6}, we need to redefine the energy norm
\[
\triple{(\bff u,v)}_{|s|}^2:=
	(\bs\sigma(\bff u),\bs\varepsilon(\overline{\bff u}))_{\Omega_-}
	+  \| s \sqrt{\rho_\Sigma} \bff u\|_{\Omega_-}^2 
	+\rho_f \| \nabla v\|_{\mathbb R^d\setminus\Gamma}^2
	+\rho_f  \| (s/c)\,v\|_{\mathbb R^d}^2,
\]
due to the fact that the elastic field is not handled with a potential representation and, therefore, it does not extend to the other side of the interface. Note that $\triple{\cdot}_1$ is equivalent to the $\mathbf H^1(\Omega_-)\times H^1(\mathbb R^d\setminus\Gamma)$ norm and that, similarly to \eqref{eq:3.9}, 
\begin{equation}\label{eq:4.9}
\underline\sigma\triple{(\mathbf u, v)}_1\leq \triple{(\mathbf u,v)}_{|s|}\leq\frac{|s|}{\underline \sigma}\triple{(\mathbf u, v)}_1.
\end{equation}

\begin{proposition}[Well-posedness]\label{prop:4.3}
Problem \eqref{eq:4.6} is uniquely solvable for any $(\phi_0,\lambda_0)\in H^{1/2}(\Gamma)\times H^{-1/2}(\Gamma)$ and $s\in \mathbb C_+$. Moreover, there exist $C_1,C_2>0$, independent of $h$, such that
\begin{align} 
\label{eq:4.10}
\triple{(\bff u^h,v^h)}_1+\|\phi^h\|_{1/2,\Gamma}
	\le& C_1 \frac{|s|}{\sigma\underline\sigma^2 } \| (s\phi_0,\lambda_0)\|_{1/2,-1/2,\Gamma}, \\
\label{eq:4.11}
\|\lambda^h\|_{-1/2,\Gamma} \leq & C_2\frac{|s|^{3/2}}{ \sigma\underline\sigma^{3/2}}\| (s\phi_0,\lambda_0)\|_{1/2,-1/2,\Gamma}.
\end{align}
\end{proposition}
\begin{proof}
It is straightforward to verify that
\begin{equation}\label{eq:4.12}
\mathrm{Re}\left(\overline{s}\mathcal{A}\left((\mathbf{u},v),(\overline{\mathbf{u}},\overline v);s\right) \right)=\sigma\triple{(\mathbf{u},v)}_{|s|}^2,
\end{equation}
and
\[
|f\left((\mathbf{w},w);s\right)|\leq C\|(s\phi_0,\lambda_0)\|_{1/2,-1/2,\Gamma}\triple{(\mathbf{w},w)}_1,
\]
where the constant depends only on $\rho_f$ and $\Gamma$.
Hence, by \eqref{eq:4.9} and the Lax-Milgram lemma, we have unique solvability of \eqref{eq:4.6} and the following bound in the energy norm:
\begin{equation}
\label{eq:4.13}
\triple{(\bff u^h,v^h)}_{|s|} \le  C \frac{|s|}{\sigma \underline\sigma} 
	\|(s\phi_0,\lambda_0)\|_{1/2,-1/2,\Gamma}.
\end{equation}
The estimate \eqref{eq:4.10} can be easily derived from \eqref{eq:4.13} and \eqref{eq:4.9} and the fact that
 $\phi^h=-\jump{\gamma v^h}$. Finally, recalling that $\lambda^h=-\jump{\partial_\nu v^h}$ and using \cite[Lemma 15]{LaSa:2009a}, namely if $\Delta v-s^2 v=0$ in an open set $\mathcal O$ with Lipschitz boundary, then
\begin{equation}\label{eq:4.14}
\|\partial_\nu v\|_{-1/2,\partial\mathcal O}\leq C\left(\frac{|s|}{\underline\sigma}\right)^{1/2}
(\| s v\|_{\mathcal O}+\|\nabla v\|_{\mathcal O}),
\end{equation}
it can be shown that \eqref{eq:4.11} follows from \eqref{eq:4.10}.
\end{proof}
%
\subsection{Semidiscretization error}
We now study the difference between the solutions to the exact problem and their finite dimensional approximations. It is important to stress that $\mathbf{u}^h-\mathbf{u} \notin \mathbf V_h$, and therefore we will not be able to write a transmission problem for the error $\mathbf u^h-\mathbf u$ in the style of  \eqref{eq:4.4}. Instead, we will work with the difference
\[
\mathbf e^h:=\mathbf u^h -\mathbf P_h \mathbf u,
\]
where $\mathbf P_h:\mathbf H^1(\Omega_-)\to \mathbf V_h$ is an elliptic projection that will be defined below. We first need to introduce the finite dimensional space of rigid motions
\[
\mathbf M: = \left\{ \mathbf m \in \mathbf H^1(\Omega_-): (\boldsymbol\sigma(\mathbf m),\boldsymbol\varepsilon(\mathbf m))_{\Omega_-}=0\right\}.
\] 
From now on we will assume that $\mathbf M\subset \mathbf V_h$. The operator $\mathbf P_h$ is given by the solution of the problem
\begin{subequations}\label{eq:4.100}
\begin{alignat}{6}
(\boldsymbol\sigma(\mathbf P_h\mathbf u),\boldsymbol\varepsilon(\mathbf w))_{\Omega_-} =\,& (\boldsymbol\sigma(\mathbf u),\boldsymbol\varepsilon(\mathbf w))_{\Omega_-}\, && \quad \forall \mathbf w \in \mathbf V_h, \\
(\mathbf P_h\mathbf u,\mathbf m)_{\Omega_-} =\,& (\mathbf u,\mathbf m)_{\Omega_-} && \quad \forall \mathbf m\in \mathbf M.
\end{alignat}
\end{subequations}
Using Korn's inequality it is easy to show that $\mathbf P_h$ is well defined and that the approximation error $\|\mathbf u-\mathbf P_h\mathbf u\|_{1,\Omega_-}$ is equivalent to the $\mathbf H^1(\Omega_-)$-best approximation on $\mathbf V_h$. In order to shorten notation, we will write $\mathbf r^h:=\mathbf P_h\mathbf u-\mathbf u$. 

The triplet $(\mathbf e^h, \phi^h ,\lambda^h)\in\mathbf V_h\times Y_h\times X_h$ satisfies the following error equations:
\begin{subequations}\label{eq:4.15}
\begin{align}
\label{eq:4.15a}
a(\mathbf e^h,\mathbf w;s)_{\Omega_-} +s^2\left(\rho_{\Sigma}\mathbf r^h,\mathbf{w}\right)_{\Omega_-} +\rho_{f}s\langle(\phi^h-\phi),\gamma\mathbf{w}\cdot\boldsymbol\nu\rangle_{\Gamma} & = 0 
	\quad\forall \mathbf w\in \mathbf V_h\\
\label{eq:4.15b}
-s\gamma(\mathbf{e}^h+\mathbf r^h)\cdot\boldsymbol\nu + \mathrm{W}(s/c)(\phi^h-\phi) - \left(\tfrac{1}{2}\mathrm{I}-\mathrm{K}^{t}(s/c)\right)(\lambda^h-\lambda) & \,\in\, Y_h^{\circ},\\
\label{eq:4.15c}
\left(\tfrac{1}{2}\mathrm{I}-\mathrm{K}(s/c)\right)(\phi^h-\phi) + \mathrm{V}(s/c)(\lambda^h-\lambda) &  \,\in\,X_h^{\circ}.
\end{align}
\end{subequations}
For this system there is a corresponding non standard transmission problem\\

\begin{proposition}\label{prop:4.4}
If $(\mathbf{e}^h,\lambda^h,\phi^h)$ satisfies \eqref{eq:4.15} and we define
\[
e^h := \mathrm D(s/c)(\phi^h-\phi)-\mathrm S(s/c)(\lambda^h-\lambda),
\]
then the pair then $(\mathbf{e}^h,e^h)$ is a solution of the transmission problem
\begin{subequations}\label{eq:4.16}
\begin{alignat}{6}
\label{eq:4.16a}
a(\mathbf e^h,\mathbf w;s)_{\Omega_-}  -s\langle\rho_{f}\jump{\gamma e^h},\gamma\mathbf{w}\cdot\boldsymbol\nu\rangle_{\Gamma}  = & -s^2\left(\rho_{\Sigma}\mathbf r^h,\mathbf{w}\right)_{\Omega_-} &
	\quad\forall \mathbf w\in \mathbf V_h, \\
\label{eq:4.16b}
-\Delta e^h +(s/c)^2e^h =&\,0 &\quad \mbox{in\; $\mathbb{R}^d\setminus\Gamma$},\\
\label{eq:4.16c}
\jump{\gamma e^h}-\phi \,\in &\, Y_h,\\
\label{eq:4.16d}
\jump{\partial_{\nu}e^h}-\lambda \,\in &\,X_h,\\
\label{eq:4.16e}
s\gamma(\mathbf{e}^h+\mathbf r^h)\cdot\boldsymbol\nu+ \partial^+_\nu e^h\,\in &\, Y_h^{\circ},\\
\label{eq:4.16f}
\gamma^-e^h \,\in &\, X_h^{\circ}.
\end{alignat}
\end{subequations}
Conversely, if $(\mathbf{e}^h,e^h)$ is a solution of \eqref{eq:4.16} then
\[
(\mathbf{e}^h,\phi^h,\lambda^h) := (\mathbf{e}^h,\phi-\jump{\gamma e^h},\lambda-\jump{\partial_{\nu}e^h}),
\]
solve \eqref{eq:4.15}.
\end{proposition}
\begin{proof}
Starting with a solution of \eqref{eq:4.15}, we see that equation \eqref{eq:4.16b} is a consequence of the definition of $e^h$, while \eqref{eq:4.16a} follows readily from \eqref{eq:4.15a}, noting that $(\jump{\gamma e^h},\jump{\partial_\nu e^h})=(\phi-\phi^h,\lambda-\lambda^h)$. The equations \eqref{eq:4.16c} and \eqref{eq:4.16d} can also be verified from the last observation, since
$
 Y_h\times X_h\ni (\phi^h,\lambda^h)=(\phi-\jump{\gamma e^h},\lambda-\jump{\partial_{\nu}e^h}).
$
Finally, using
\[
\partial_{\nu}^-(\mathrm S(s)\lambda) = (\tfrac{1}{2}\mathrm I +\mathrm{K}^t(s))\lambda \,,\qquad \gamma^-(\mathrm D (s)\phi) = \left(-\tfrac{1}{2}\mathrm I + \mathrm{K}(s)\right)\phi,
\]
we see that \eqref{eq:4.15b} and \eqref{eq:4.15c} imply \eqref{eq:4.16e} and \eqref{eq:4.16f}.

The proof of the converse statement is very similar.
\end{proof}
\begin{proposition}\label{prop:4.5}
The system \eqref{eq:4.16} is equivalent to the variational problem of finding $(\mathbf{e}^h,e^h)\in \mathbf{H}^1(\Omega_-)\times H^1(\mathbb{R}^d\setminus\Gamma)$ such that
\begin{subequations}\label{eq:4.17}
\begin{alignat}{6}
\label{eq:4.17a}
 (\gamma^-e^h,\jump{\gamma e^h}-\phi) \in \; &  X_h^{\circ}\times Y_h  , && \\
\label{eq:4.17b}
\mathcal{A}((\bff e^h, e^h),(\bff w,w);s)  =\; &b\left((\mathbf w, w);s\right) &&
\qquad\forall (\bff w,w)\; \in \mathbf{V}_h\times V_h,
\end{alignat}
\end{subequations}
where the bilinear form $\mathcal{A}$ is defined in the statement of Proposition \ref{prop:4.2} and
\[
b\left((\mathbf w,w);s\right):= \rho_f\langle\lambda,\gamma^-w\rangle_{\Gamma} +s\rho_f\langle\gamma \mathbf r^h\cdot\boldsymbol\nu,\jump{\gamma w}\rangle_{\Gamma}-s^2\left(\rho_{\Sigma}\mathbf r^h,\mathbf{w}\right)_{\Omega_-} .
\]
\end{proposition}
\begin{proof}
The proof is very similar to the one of Proposition \ref{prop:4.2}. Details are omitted. \end{proof}

\begin{proposition}\label{prop:4.6}
Problem \eqref{eq:4.17} is uniquely solvable for any $(\mathbf u,\phi,\lambda)\in \mathbf H^1(\Omega_-)\times H^{1/2}(\Gamma)\times H^{-1/2}(\Gamma)$ and $s\in \mathbb C_+$. Moreover, there exist constants $C_1, C_2>0$ independent of $h$ such that
\begin{align}
\label{4.18}
\triple{(\bff e^h,e^h)}_{1} + \|\phi-\phi^h\|_{1/2,\Gamma} \le& C_1\frac{|s|}{\sigma\underline{\sigma}}\Big( 
\| (s\,\phi,\lambda)\|_{1/2,-1/2,\Gamma}
+\|s\mathbf r^h\|_{1,\Omega_-}+\|s^2\mathbf r^h\|_{\Omega_-}\Big), \\
\label{4.19}
\|\lambda-\lambda^h\|_{1/2,\Gamma}\le& C_2\frac{|s|^{3/2}}{\sigma\underline{\sigma}^{3/2}}\Big(\| (s\,\phi,\lambda)\|_{1/2,-1/2,\Gamma}+\|s\mathbf r^h\|_{1,\Omega_-}+\|s^2\mathbf r^h\|_{\Omega_-}\Big).
\end{align}
\end{proposition}
\begin{proof}
The existence and uniqueness of the solution to \eqref{eq:4.17} is proven in a way analogous to that used in Proposition \ref{prop:4.3}. We will next prove a bound in the energy norm
\begin{equation}
\label{eq:4.20}
\triple{(\mathbf e^h,e^h)}_{|s|} \leq C_1\frac{|s|}{\sigma\underline{\sigma}}\Big(\| (s\,\phi,\lambda)\|_{1/2,-1/2,\Gamma}+\|s\mathbf r^h\|_{1,\Omega_-}+\|s^2\mathbf r^h\|_{\Omega_-}\Big).
\end{equation}
The estimate \eqref{4.18} follows from \eqref{eq:4.20} and \eqref{eq:4.9}. In order to get to \eqref{4.19} we make use of \eqref{eq:4.20}, the fact that $\lambda-\lambda^h=\jump{\partial_\nu e^h}$, and \eqref{eq:4.14}. 

To prove \eqref{eq:4.20} we proceed as follows. We first obtain an upper bound for the bilinear form
\begin{equation}
\label{eq:4.21}
|\mathcal{A}((\bff u,v),(\bff w,w);s)| 
	\le 
	C \frac{|s|}{\underline\sigma} 
		\triple{(\bff u,v)}_1\triple{(\bff w,w)}_{|s|},
\end{equation}
by the same argument that was used in Proposition \ref{prop:3.7}. Also
\begin{equation}
\label{eq:4.22}
|b\left((\mathbf w, w);s\right)| \leq \frac{C}{\underline\sigma}\left(\|\lambda\|_{-1/2,\Gamma}+\|s\mathbf r^h\|_{1,\Omega_-} +\|s^2\mathbf r^h\|_{\Omega_-}\right) \triple{(\mathbf w, w)}_{|s|}.
\end{equation}
The constants in \eqref{eq:4.21} and \eqref{eq:4.22} depend only on the geometry. Now, for $\phi\in H^{1/2}(\Gamma)$, pick a lifting $w_\phi\in H^1(\mathbb R^d\setminus\Gamma)$ such that $\gamma^+w_\phi=\phi$,  $\gamma^-w_\phi=0$, and
\begin{equation}\label{eq:4.23}
 \|w_\phi\|_{1,\mathbb{R}^d\setminus\Gamma}\leq C \|\phi\|_{1/2,\Gamma}.
\end{equation}
Since $(\mathbf e^h,e^h+w_\phi)\in\mathbf V_h\times V_h$ we can use \eqref{eq:4.12}, \eqref{eq:4.17b}, \eqref{eq:4.21}, and \eqref{eq:4.22} (i.e., coercivity, the variational equation, and boundedness of the bilinear form and right-hand side) to estimate
\begin{align*}
\triple{(\mathbf e^h,e^h+w_\phi)}_{|s|}^2 \leq & \frac{|s|}{\sigma}|\mathcal A \left((\mathbf e^h,e^h+w_\phi),(\mathbf e^h,e^h+w_\phi);s\right)| \\
 = & \frac{|s|}{\sigma}|b\left((\mathbf e^h,e^h+w_\phi);s\right) + \mathcal A \left((\mathbf 0,w_\phi),(\mathbf e^h,e^h+w_\phi);s\right)| \\
\leq & C\frac{|s|}{\sigma\underline\sigma}  \triple{(\mathbf e^h,e^h+w_\phi)}_{|s|}  \\
& \left( |s| \|w_\phi\|_{1,\mathbb{R}^d\setminus\Gamma}+ \|\lambda\|_{-1/2,\Gamma}+\|s\mathbf r^h\|_{1,\Omega_-}+\|s^2\mathbf r^h\|_{\Omega_-}\right).
\end{align*}
This bound, together with
\[
\triple{(\mathbf 0,\omega_\phi)}_{|s|}\le \frac{C}{\underline\sigma}\| s\,\phi\|_{1/2,\Gamma}
\]
(see \eqref{eq:4.9} and \eqref{eq:4.23}) prove \eqref{eq:4.20}.
\end{proof}
%
\subsection{Estimates in the time-domain}
Using the bounds for  the error operators derived in the previous section, we can prove explicit time-domain estimates. 
Just like in the BEM/BEM case,  we can use  \cite[Theorem 7.1]{DoSa:2013} and combine it with the Laplace-domain estimates from Propositions \ref{prop:4.3} and \ref{prop:4.6} to obtain the following results.
\begin{corollary}\label{cor:4.7}
Consider causal problem data
$\phi_0\in W^4_+(H^{1/2}(\Gamma))$ and $\lambda_0\in  W^3_+(H^{-1/2}(\Gamma)).$
Then $\mathbf u^h, v^h, \phi^h, \lambda^h$ are continuous causal functions of time and for all $t\geq0$ :
\begin{align*}
\triple{(\mathbf{u}^h,v^h)(t)}_1 +\|\phi^h(t)\|_{1/2,\Gamma}\leq\, & D_1\max\{1,t^2\}\frac{t^2}{t+1}\int_0^t\|\mathcal{P}_{3}(\dot{\phi_0},\lambda_0)(\tau)\|_{1/2,-1/2,\Gamma}\;d\tau, \\
\|\lambda^h(t)\|_{-1/2,\Gamma} \leq\, & D_2\max\{1,t^{3/2}\}\frac{t\sqrt t}{\sqrt{t+1}}\int_0^t\|\mathcal{P}_{3}(\dot{\phi_0},\lambda_0)(\tau)\|_{1/2, -1/2,\Gamma}\;d\tau.
\end{align*}
where $D_1$ and $D_2$ depend only on $\Gamma$.
\end{corollary}
To abbreviate the following statements, we will use the following shorthand for approximation errors
\begin{align*}
a_h(t):=\;& \int_0^t\left( \|\mathcal P_3(\dot\phi-\Pi_h^Y\dot{\phi}^h)(\tau)\|_{1/2,\Gamma} +\|\mathcal P_3(\lambda-\Pi_h^X\lambda^h)(\tau)\|_{-1/2,\Gamma} \right) \,d\tau\\
	& + \int_0^t\left( \|\mathcal P_3(\dot{\mathbf u}-\mathbf P_h\dot{\mathbf u}^h)(\tau)\|_{1,\Omega_-} + \|\mathcal P_3(\ddot{\mathbf u}-\mathbf P_h\ddot{\mathbf u}^h)(\tau)\|_{\Omega_-}\right) \,d\tau,
\end{align*}
where $\Pi_h^Y:H^{1/2}(\Gamma)\to Y_h$ and $\Pi_h^X:H^{-1/2}(\Gamma)\to X_h$ are orthogonal projections and $\mathbf P_h$ is the elliptic elastic projection onto $\mathbf V_h$ defined in \eqref{eq:4.100}.
\begin{corollary}\label{cor:4.8}
If  the solution triplet satisfies
\[
(\mathbf u,\phi ,\lambda) \in W_+^3(\mathbf H^1(\Omega_-))\times W^4_+(H^{1/2}(\Gamma))\times W^3_+(H^{-1/2}(\Gamma)),
\]
then $(\mathbf{e}^h,e^h) \in\mathcal{C}(\mathbb{R},\mathbf{H}^1(\Omega_-)\times H^1(\mathbb{R}^d\setminus\Gamma))$ is causal and we have constants $D_1$ and $D_2$ depending only on $\Gamma$ such that  for $t\geq0$
\begin{align*}
\triple{(\mathbf{e}^h,e^h)(t)}_1 +\|(\phi-\phi^h)(t)\|_{1/2,\Gamma} \leq \; & D_1\max\{1,t\}\frac{t^2}{t+1}\; a_h(t), \\
\\
\|(\lambda-\lambda^h)(t)\|_{-1/2,\Gamma} \leq \; & D_2\max\{1,t^{3/2}\}\frac{t^{3/2}}{\sqrt{t+1}}\; a_h(t).
\end{align*}
\end{corollary}
%
\paragraph{Full discretization with BDF2-CQ}
%

The purely boundary integral formulation treated in the first part of this paper lent itself naturally to a full discretization using one of the many Convolution Quadrature schemes for the time evolution. For the current variational/boundary integral formulation it would seem that an independent treatment with traditional time-stepping for the Finite Element part and Convolution Quadrature for the discretized boundary integral equations would be the best way to proceed, and for our computational implementation we will proceed in this fashion.

However, it turns out that the separate application of time stepping and CQ to different parts of the system is equivalent to the application of CQ globally, as long as the time stepping method used for the FEM part coincides with the one giving rise to the CQ family used for the implementation (see \cite[Proposition 12]{LaSa:2009a}, \cite{HaSa:2015}). This observation will allow us to analyze the fully discrete method as if the whole discretization were done with CQ. 

We present results for the coupled schemes based on BDF2. In the following section we will show numerical experiments for BDF2-CQ and Trapezoidal Rule-CQ. (We note that the analysis of Trapezoidal Rule CQ was done by Lehel Banjai in \cite{Banjai:2010}, although it does not give explicit behaviour of bounds with respect to $t$.) We will use $(\mathbf u_\kappa^h,v_\kappa^h)$ to denote the fully discrete approximation of $(\mathbf u,v)$ using a CQ method with constant time-step $\kappa$.
In parallel to the corresponding result in Section \ref{sec:3} (Proposition \ref{prop:3.11}), the next estimate follows from the Laplace-domain estimates in Proposition \ref{prop:4.6} and an application of \cite[Proposition 4.6.1]{Sayas:2014}. 

\begin{proposition}\label{prop:4.9}
Let $\ell=6$ and $(\phi_0,\lambda_0)$ be causal problem data such that
\[
(\phi_0,\lambda_0)\in W^{\ell+1}_+(H^{1/2}(\Gamma))\times W^\ell_+(H^{-1/2}(\Gamma)).
\]
Then, for $t\geq0$, it holds that
\[
\triple{(\mathbf u^h,v^h)(t)-(\mathbf u^h_\kappa,v^h_\kappa)(t)}_1 \leq D\kappa^2(1+t^2)\int_0^t\|(\phi_0^{(\ell+1)},\lambda_0^{(\ell)})(\tau)\|_{1/2,-1/2,\Gamma}\,d\tau.
\]
\end{proposition}
It is important to note that the high-order regularity $\ell=6$ is only required to achieve optimal convergence of order $\kappa^2$. For problem data with regularity as low as $\ell=3$, reduced convergence of order $\kappa^{3/2}$ is achieved (see \cite[Chapter 4]{Sayas:2014}).

%
\section{Numerical Experiments}

In this section, we show some experiments for fully discrete methods applied to the BEM and BEM/FEM formulations  we have analyzed. For general ideas of what CQ time-discretization means and how it is used, we refer to \cite{BaSc:2012, HaSa:2014, DoLuSa:2014b}. Algorithms for BEM/FEM applied to acoustic transmission problems are explained in \cite{HaSa:2015}. 

\subsection{Boundary integral method}
In order to test the convergence properties of the implementation the following synthetic problem was solved in $\mathbb{R}^2$. The interior elastic domain will be the unit disk $\Omega_-=\{ \mathbf x\,:\, x_1^2+x_2^2<1\}$, and its exterior will be the acoustic domain. If we let $\mathcal{H}(t)$ be a smooth approximation to the Heaviside function, then the elastic causal pressure wave
\[
\mathbf{u}(\mathbf{x},t) = \mathcal{H}(c_Lt-\mathbf{x}\cdot\mathbf{d})\sin\left(3(c_Lt-\mathbf{x}\cdot\mathbf{d})\right)\mathbf{d}, \quad \mathbf{d}=\left(\tfrac{1}{\sqrt{2}},\tfrac{1}{\sqrt{2}}\right), \quad c_L = \sqrt{\tfrac{2\mu+\lambda}{\rho}},
\]
and the cylindrical acoustic wave
\[
v(\mathbf x,t)=\mathcal{L}^{-1}\left\{\imath H^{(1)}_0(\imath s |\mathbf x|)\,\mathcal{L}\{\mathcal{H}(t)\sin(2t)\} \right\}
\]
solve equations \eqref{eq:2.2a}  and \eqref{eq:2.2b} respectively. Here $\mathcal L$ is the Laplace transform. In order for them to satisfy the entire IBVP \eqref{eq:2.2}, equations  \eqref{eq:2.2c} and  \eqref{eq:2.2d} were used to define the boundary data $\alpha_0 :=\partial_{\nu}v^{inc}$ and $\beta_0:=v^{inc}$. 

The boundary data was sampled accordingly and the Laplace transformed equivalent system \eqref{eq:2.4} was discretized in space with {\tt deltaBEM} (the reader is referred to \cite{DoSaSa:2015,DoLuSa:2014b} for further details on the computational aspects of {\tt deltaBEM}), which can be considered as a Galerkin $\mathcal P_1$ method with reduced quadrature, while Convolution Quadrature was used for time stepping on increasingly finer space/time discretizations with $N$ space points and $M$ time steps. The approximated solutions were then sampled in 20 random points on the circle of radius $r=.7$ for the elastic wave and $r=2$ for the acoustic wave and compared against the exact solutions. The maximum difference in the final time
\begin{align*}
E^{v}_{h,k}:=\;& \frac{\max_{i=1}^{20}|v(\mathbf x_i,t_f)-v^{h,k}(\mathbf x_i,t_f)|}{\max_{i=1}^{20}|v(\mathbf x_i,t_f)|}, \\
E^{\mathbf u}_{h,k}:=\;& \frac{\max_{i=1}^{20}|\mathbf u(\mathbf x_i,t_f)-\mathbf u^{h,k}(\mathbf x_i,t_f)|}{\max_{i=1}^{20}|\mathbf u(\mathbf x_i,t_f)|},
\end{align*}
is used as the error measure. Trapezoidal Rule CQ and BDF2-CQ were both implemented and compared. Tables \ref{tab:1} and \ref{tab:2} summarize the results, while convergence plots can be seen in Figure \ref{fig:1}. In the simulations, the values $\lambda = 9$, $\mu=15$, $\rho_{\Sigma}= 1.5$, $\rho_f=1$ and $c = \sqrt{5}$ were used, the final time was $T=5$.
\begin{table}[h]\centering
\begin{tabular}{ccccc}
\hline
\multicolumn{1}{|c|}{$N/M$} & \multicolumn{1}{c|}{$E^{\mathbf u}_{h,k}$} & \multicolumn{1}{c|}{e.c.r.} & \multicolumn{1}{c|}{$E^{v}_{h,k}$} & \multicolumn{1}{c|}{e.c.r.}  \\ \hline
 45/90  	&   0.8745  & ---    & 1.1603  & ---       \\ \hline
 60/120 	&   0.7131  & 1.2265 & 0.9862  & 1.1766    \\ \hline
 90/180 	&   0.3692  & 1.9312 & 0.8900  & 1.1080    \\ \hline
 120/240 	&   0.2022  & 1.8265 & 0.4778  & 1.8627   \\ \hline
 180/360	&   0.0806  & 2.5079 & 0.2407  & 1.9854   \\ \hline
 240/480 	&   0.0466  & 1.7285 &  0.1482 & 1.6240  \\ \hline
 360/720 	&   0.0302  & 1.5456 &  0.0513 & 2.8869             
\end{tabular}
\caption{Relative errors and estimated convergence rates in the time-domain for the BDF2 Convolution Quadrature with lowest order Galerkin discretization (with reduced quadrature): $N$ represents the number of space discretization points (elements), $M$ is the number of timesteps. The errors are measured at the final time  $T=5$.}\label{tab:1}
\end{table}
\begin{table}[h]\centering
\begin{tabular}{ccccc}
\hline
\multicolumn{1}{|c|}{$N/M$} & \multicolumn{1}{c|}{$E^{\mathbf u}_{h,k}$} & \multicolumn{1}{c|}{e.c.r.} & \multicolumn{1}{c|}{$E^{v}_{h,k}$} & \multicolumn{1}{c|}{e.c.r.}  \\ \hline
 45/90  	&   9.9416  & ---    & 1.0642  & ---       \\ \hline
 60/120 	&  73.3473  &  0.135 & 0.3961  & 2.6864    \\ \hline
 90/180 	&   0.1402  & 523.3082 & 0.2089  & 1.8962  \\ \hline
 120/240 	&   0.0675  & 2.0755 & 0.1261  & 1.6571    \\ \hline
 180/360	&   0.0484  & 1.3955 & 0.0522  & 2.4138    \\ \hline
 240/480 	&   0.0252  & 1.9219 &  0.0308 & 1.6968    \\ \hline
 360/720 	&   0.0181  & 1.3923 &  0.0126 & 2.4404    \\ \hline
 480/960        &   0.0099  & 1.8212 &  0.0066 & 1.9054
\end{tabular}
\caption{Relative errors and estimated convergence rates in the time-domain for the Trapezoidal Rule Convolution Quadrature with the same space discretization as in Table \ref{tab:1}: $N$ represents the number of space discretization points, $M$ is the number of timesteps. The errors are measured at the final time  $T=5$.}\label{tab:2}
\end{table}
\begin{figure}[h]\centering
\includegraphics[scale=.4]{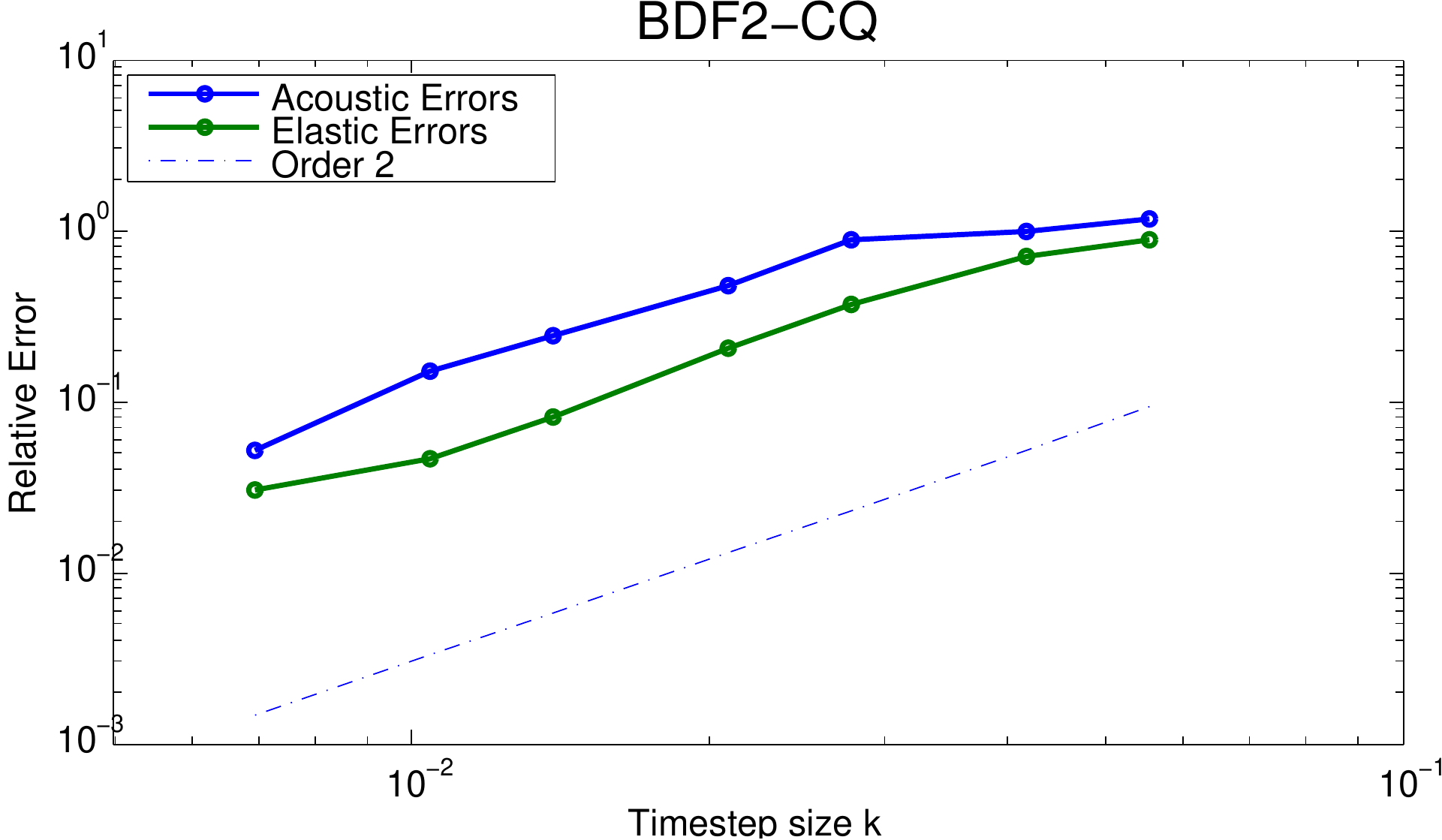}
\includegraphics[scale=.4]{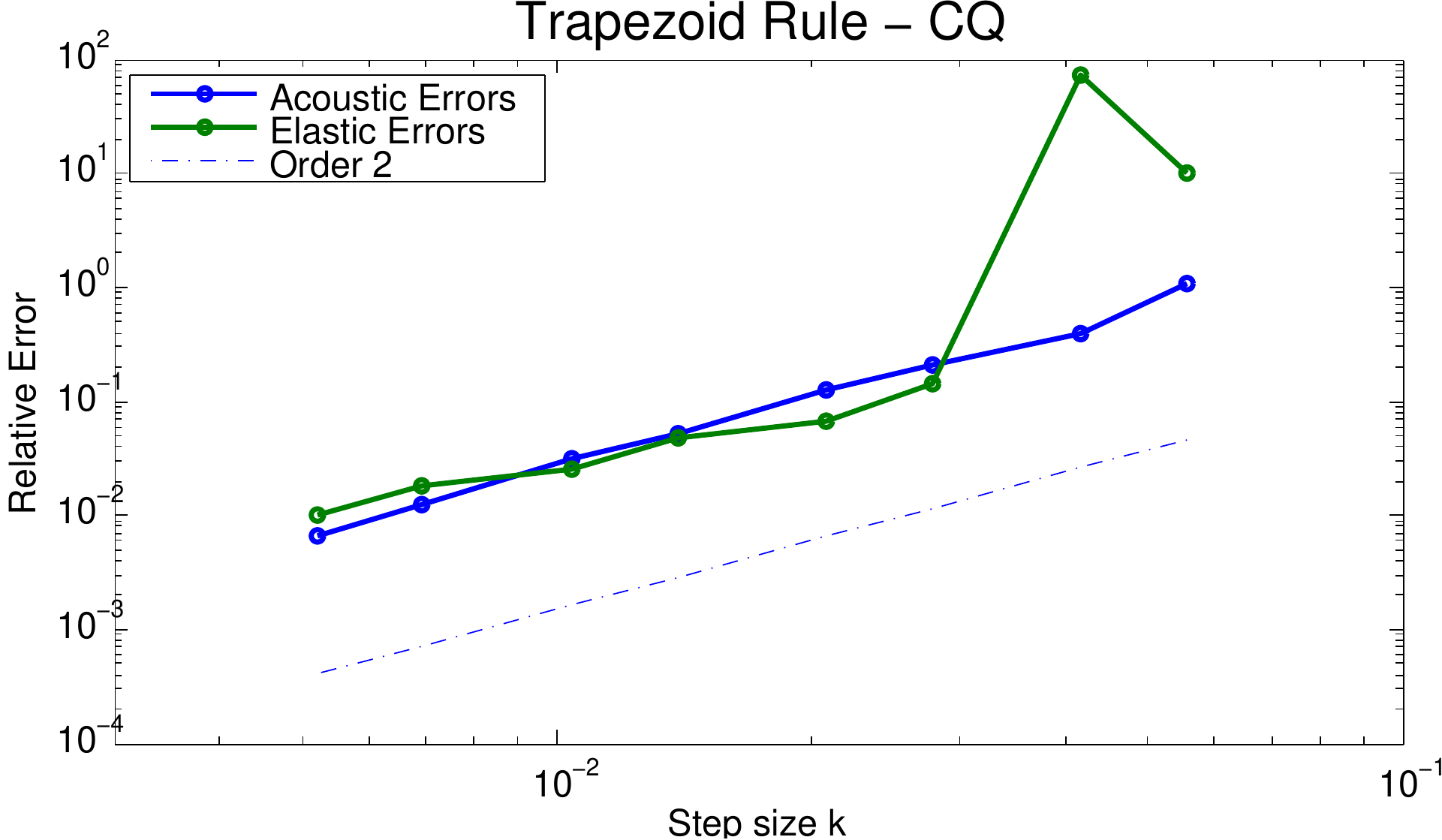}
\caption{Relative errors for the BDF2 and TR implementations of CQ. The maximum difference between the approximate and exact solutions on the sampled points.}\label{fig:1}
\end{figure}
\subsection{Coupled boundary-field method}
%

The previous coupling scheme was implemented using $\mathcal P_3/\mathcal P_2$ Boundary Elements for the acoustic wave field and $\mathcal P_3$ Finite Elements for the interior elastic wavefield. The convergence studies were carried out using the rectangle $[1,3]\times [1,2]$ as the elastic domain, which was triangulated using Matlab-produced unstructured meshes. Known solutions were imposed for the interior and exterior problems; a plane pressure wave on the interior
\[
\mathbf{u} = \psi(c_Lt-\mathbf{x}\cdot\mathbf{d})\,\mathbf{d},\qquad \psi(t):= \mathcal{H}(t)\sin(2t),\qquad c_L:=\sqrt{\tfrac{\lambda+2\mu}{\rho}},
\]
and a cylindrical acoustic wave on the exterior
\[
v = \mathcal{L}^{-1}\left\{\tfrac{i}{4}H^{(1)}_{0}(3|\mathbf{x}-\mathbf{x}_0|)\,\mathcal{L}\{\varphi(t)\}\right\},\qquad \varphi(t):= \mathcal{H}(t)\sin(3t) 
\]
where $\mathbf{x}_0 = (1.5,1.5)$ is the location of the source of the cilyndrical wave, $\mathcal{H}(t)$ is a smooth approximation to the Heaviside function, and $\lambda = 2$, $\mu=3$ and $\rho=5$. 

These two functions satisfy equations \eqref{eq:2.3a} and \eqref{eq:2.3b}. In order to force them to solve the problem in question, the boundary data was manufactured using \eqref{eq:2.3c} and \eqref{eq:2.3b} as the definitions for $(\lambda_0,\phi_0)$. The relevant information was sampled from the known solution, combined according to \eqref{eq:2.3c} and \eqref{eq:2.3b}  and the resulting pair $(\lambda_0,\phi_0)$ was then fed to the discrete system as boundary data.

The experiment was run with fixed FEM and BEM grids with $h=0.025$ (maximum element area $3.5\times10^{-4}$ for the FEM mesh) for a final time $T=1.5$. Trapezoidal Rule time stepping and Trapezoidal Rule Convolution Quadrature was used respectively for the Finite Element and Boundary Element domains with doubling number of time steps starting at 5 and all the way up to 160. The errors were measured for the final time, for the finite element solution $E^{\mathbf{u}}_{h,k,L^2}$ in the $L^2(\Omega_-)$ norm and $E^{\mathbf{u}}_{h,k,H^1}$ in the $H^1(\Omega_-)$ norm. For the acoustic wavefield the discrete solution was postprocessed, sampled and compared to the exact solution in 10 random points in the acoustic domain, with the -normalized- maximum discrepancy $E^v_{h,k}$ being considered as the error.
\begin{figure}\centering
\includegraphics[scale=.6]{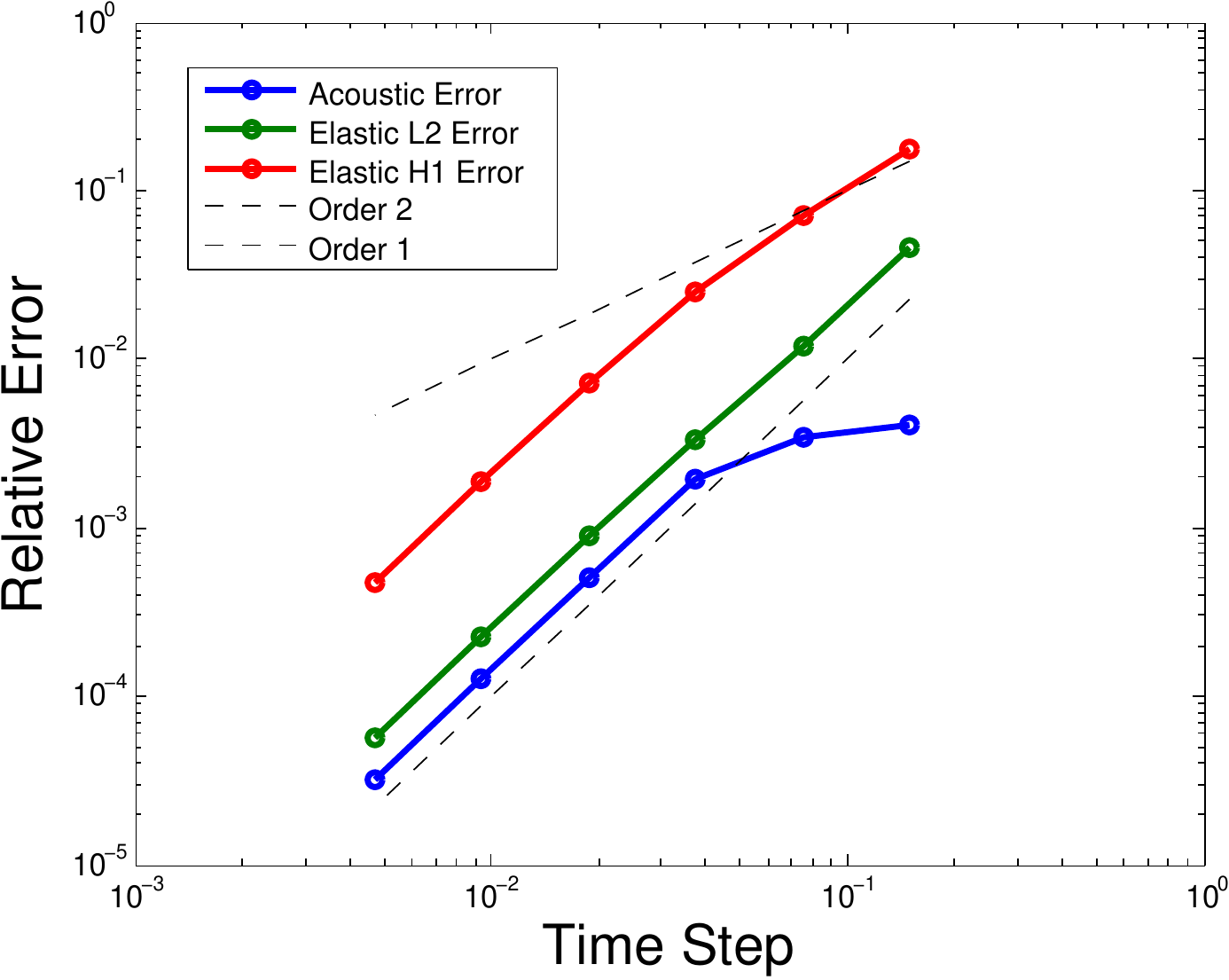}
\caption{Convergence studies for the coupled FEM/BEM scheme. $\mathcal P_3$ elements were used for the finite element solution and $\mathcal P_3/\mathcal P_2$ elements for the boundary element solution.}
\end{figure}
\begin{table}[h]\centering
\begin{tabular}{ccccccc}
\hline
\multicolumn{1}{|c|}{$M$} & \multicolumn{1}{c|}{$E^{v}_{h,k}$} & \multicolumn{1}{c|}{e.c.r.} & \multicolumn{1}{c|}{$E^{\mathbf{u}}_{h,k,L^2}$} & \multicolumn{1}{c|}{e.c.r.} & \multicolumn{1}{c|}{$E^{\mathbf{u}}_{h,k,H^1}$} & \multicolumn{1}{c|}{e.c.r.}  \\ \hline
 5  	&   4.051 E-3  & ---   & 4.531 E-2  & ---   & 0.179 E-1  & ---    \\ \hline
 10 	&   3.393 E-3  & 0.255 & 1.204 E-2  & 1.912 & 7.188 E-2  & 1.315  \\ \hline
 20 	&   1.925 E-3  & 0.818 & 3.364 E-3  & 1.839 & 2.524 E-2  & 1.510  \\ \hline
 40 	&   5.108 E-4  & 1.912 & 8.800 E-4  & 1.934 & 7.184 E-3  & 1.813  \\ \hline
 80	&   1.281 E-4  & 1.996 & 2.223 E-4  & 1.981 & 1.862 E-3  & 1.947   \\ \hline
 160 	&   3.201 E-5  & 2.000 & 5.592 E-5  & 1.995 & 4.700 E-4  & 1.987
\end{tabular}
\caption{Relative errors and estimated convergence rates in the time-domain for the Trapezoidal Rule Convolution Quadrature. $M$ is the number of timesteps. Final time $T=1.5$.}\label{tab:3}
\end{table}

The last two tables and figures show the convergence studies for the Trapezoidal Rule BEM/FEM scheme with simultaneous space/time refinement. Polynomial degrees $k=1$ and $k=2$ were used.

\begin{figure}\centering
\includegraphics[scale=.55]{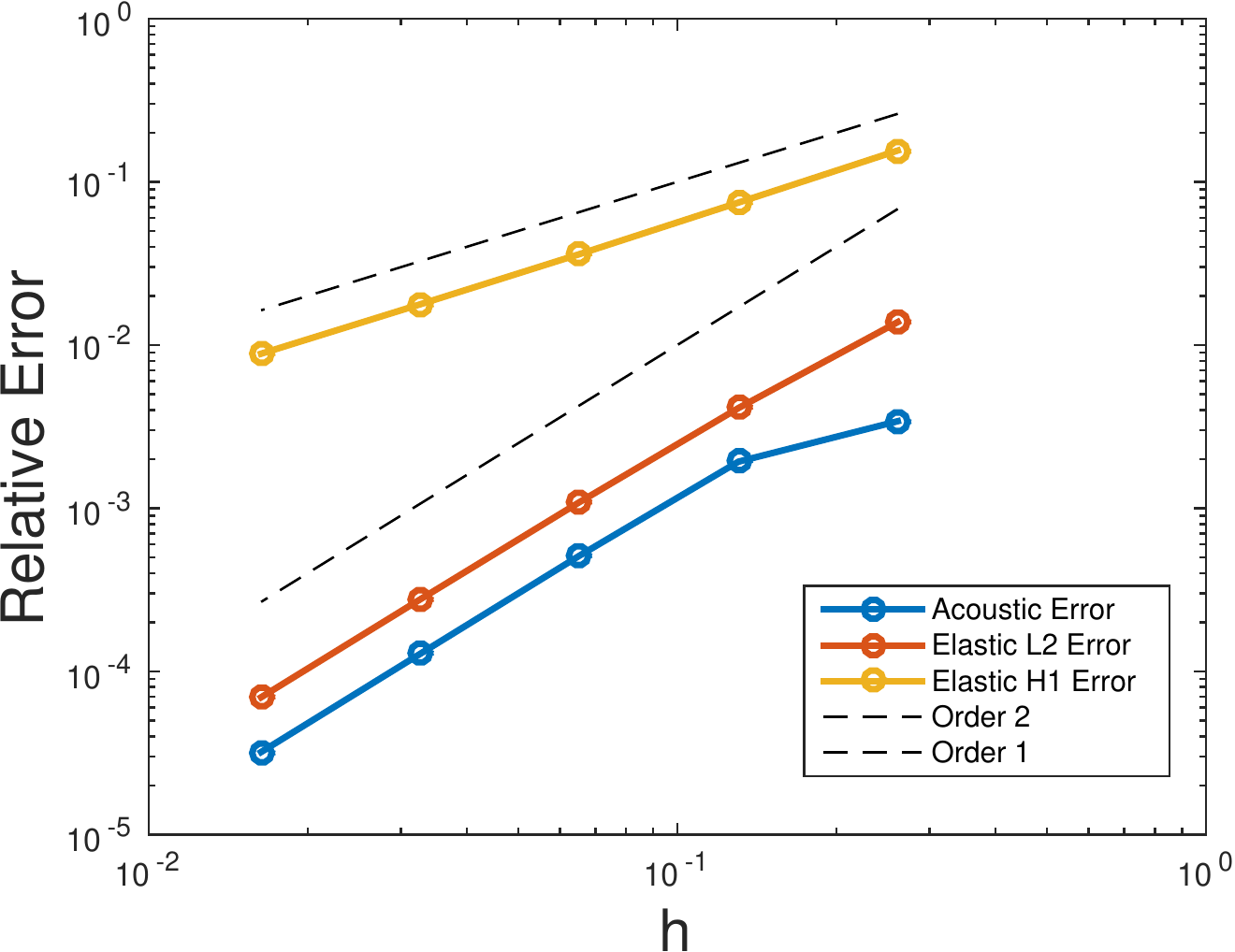}
\includegraphics[scale=.55]{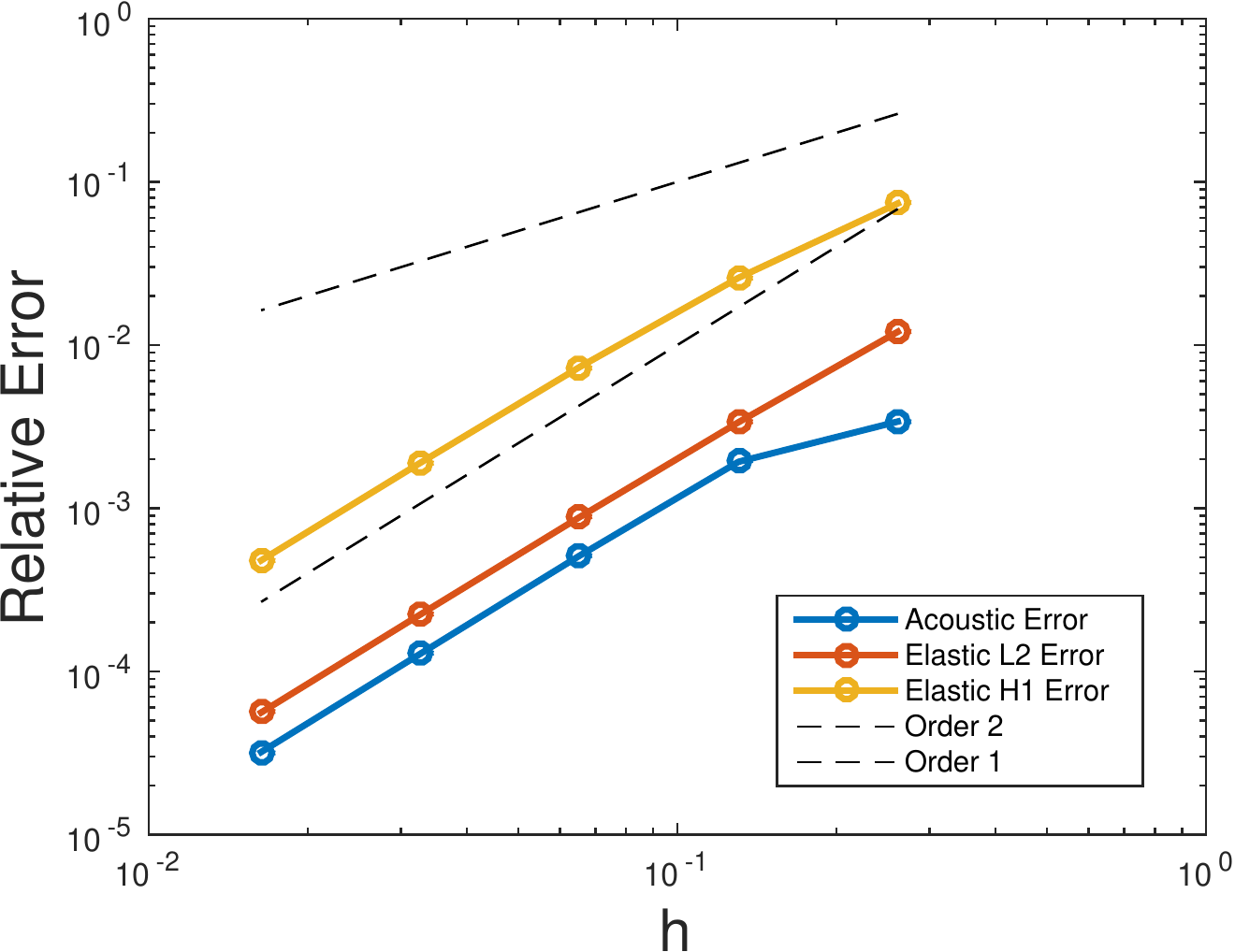}
\caption{Simultaneous space/time refinements. Left: $\mathcal P_1$ elements were used for the finite element solution and $\mathcal P_1/\mathcal P_0$ elements for the boundary element solution. Right: $\mathcal P_2$ elements were used for the finite element solution and $\mathcal P_2/\mathcal P_1$ elements for the boundary element solution.}
\end{figure}

\begin{table}[h]\centering
\begin{tabular}{ccccccc}
\hline
\multicolumn{1}{|c|}{$M/N$} & \multicolumn{1}{c|}{$E^{v}_{h,k}$} & \multicolumn{1}{c|}{e.c.r.} & \multicolumn{1}{c|}{$E^{\mathbf{u}}_{h,k,L^2}$} & \multicolumn{1}{c|}{e.c.r.} & \multicolumn{1}{c|}{$E^{\mathbf{u}}_{h,k,H^1}$} & \multicolumn{1}{c|}{e.c.r.}  \\  \hline
 10/1 	&   3.412 E-3  & ----- & 1.382 E-2  & ----- & 1.555 E-1  & -----  \\ \hline
 20/2 	&   1.929 E-3  & 0.823 & 4.112 E-3  & 1.749 & 7.467 E-2  & 1.058  \\ \hline
 40/3 	&   5.113 E-4  & 1.915 & 1.088 E-3  & 1.917 & 3.604 E-2  & 1.051  \\ \hline
 80/4	&   1.281 E-4  & 1.996 & 2.763 E-4  & 1.978 & 1.776 E-2  & 1.021   \\ \hline
 160/5 	&   3.202 E-5  & 2.000 & 6.935 E-5  & 1.994 & 8.843 E-3  & 1.006
\end{tabular}
\caption{Relative errors and estimated convergence rates in the time-domain for the Trapezoidal Rule Convolution Quadrature with $\mathcal P_1 /\mathcal P_0$ boundary elements and $\mathcal P_1$ finite elements. $h=0.52 \times 2^{-N}$ is the maximum lenght of the triangulation and $M$ is the number of timesteps. Final time $T=1.5$.}\label{tab:4}
\end{table}

\begin{table}[h]\centering
\begin{tabular}{ccccccc}
\hline
\multicolumn{1}{|c|}{$M/N$} & \multicolumn{1}{c|}{$E^{v}_{h,k}$} & \multicolumn{1}{c|}{e.c.r.} & \multicolumn{1}{c|}{$E^{\mathbf{u}}_{h,k,L^2}$} & \multicolumn{1}{c|}{e.c.r.} & \multicolumn{1}{c|}{$E^{\mathbf{u}}_{h,k,H^1}$} & \multicolumn{1}{c|}{e.c.r.}  \\ \hline
 10/1 	& 3.395 E-3  & ----- & 1.205 E-2  & ----- & 7.369 E-2  & -----  \\ \hline
 20/2 	& 1.925 E-3  & 0.819 & 3.364 E-3  & 1.841 & 2.581 E-2  & 1.513  \\ \hline
 40/3 	& 5.108 E-4  & 1.912 & 8.799 E-4  & 1.935 & 7.316 E-3  & 1.819  \\ \hline
 80/4	&   1.281 E-4  & 1.996 & 2.223 E-4  & 1.981 & 1.896 E-3  & 1.948 \\ \hline
 160/5 	&   3.201 E-5  & 2.000 & 5.591 E-5  & 1.995 & 4.783 E-4  & 1.987
\end{tabular}
\caption{Relative errors and estimated convergence rates in the time-domain for the Trapezoidal Rule Convolution Quadrature with $\mathcal P_2 /\mathcal P_1$ boundary elements and $\mathcal P_2$ finite elements. $h=0.52 \times 2^{-N}$ is the maximum lenght of the triangulation and $M$ is the number of timesteps. Final time $T=1.5$.}\label{tab:5}
\end{table}
%
\section{Conclusions}
%
We have presented stable fully discrete formulations for wave-structure scattering in the time-domain. The formulations are well suited for pure Boundary Element implementation or Boundary/Finite element coupling in the case of a general elastic scatterer. For time discretization  a BDF2 Convolution Quadrature scheme was used as a basis for the analysis, but, as the numerical experiments show, the Trapezoidal Rule schemes present good global convergence properties.

The analysis presented in this work generalizes easily to the case of several scatterers or those with non simply connected geometries. The study of scattering by obstacles with more complex physical properties such as piezoelectrics is the object of current research and could be of interest in applications such as active noise reduction.

A purely time-domain analysis, using the theory of evolution equations \cite{DoSa:2013,Sayas:2013, Sayas:2014} might shed sharper convergence estimates, with less regularity required for the solution and better control of bounds with respect to the time variable. However, the wave-structure interaction problem cannot be written as a second order evolution equation for an unbounded operator, because of the presence of time derivatives in the transmission conditions and other possible avenues have to be explored.
%
\bibliography{referencesBEM}
%

\end{document}